\documentclass{amsart}
\usepackage{amssymb}

\usepackage{amscd}
\usepackage{thmdefs}

\input{tcilatex}
\theoremstyle{definition}
\theoremstyle{remark}
\numberwithin{equation}{section}

\input tcilatex

\begin{document}
\title[regularity of homogeneous elliptic operators on $\Bbb{R}^{N}$]{$L^{p}$ regularity of homogeneous elliptic differential operators with
constant coefficients on $\Bbb{R}^{N}$ }
\author{Patrick J. Rabier}
\address{Department of mathematics, University of Pittsburgh, Pittsburgh, PA 15260}
\email{rabier@imap.pitt.edu}
\thanks{The useful comments of an anonymous referee are gratefully acknowledged.}
\maketitle

\begin{abstract}
Let $A$ be a homogeneous elliptic differential operator of order $m$ on $%
\Bbb{R}^{N}$ with constant complex coefficients. A special case of the main
result is as follows: Suppose that $u\in L_{loc}^{1}$ and that $Au\in L^{p}$
for some $1<p<\infty .$ Then, all the partial derivatives of order $m$ of $u$
are in $L^{p}$ if and only if $|u|$ grows slower than $|x|^{m}$ at infinity,
provided that growth is measured in an $L^{1}$-averaged sense over balls
with increasing radii. The necessity provides an  alternative answer to the
pointwise growth question investigated with mixed success in the literature.
Only very few special cases of the sufficiency are already known, even when $%
A=\Delta .$

The full result gives a similar necessary and sufficient growth condition
for the derivatives of $u$ of any order $k\geq 0$ to be in $L^{p}$ when $Au$
satisfies a suitable (necessary) condition. This is generalized to exterior
domains, which sometimes introduces mandatory restrictions on $N$ and $p,$
and to Douglis-Nirenberg elliptic systems whose entries are homogeneous
operators with constant coefficients but possibly different orders, as the
Stokes system.
\end{abstract}

\section{Introduction\label{intro}}

It is understood that $\Bbb{R}^{N}$ is the domain of all function spaces.
The vast PDE literature offers only surprisingly few answers to the basic
question: If $u\in \mathcal{D}^{\prime }$ (distributions) and $\Delta u\in
L^{p}$ for some $1<p<\infty ,$ what extra condition should be required of $u$
to ensure that all the second order derivatives of $u$ are in $L^{p}$?

The same question with $\Delta $ replaced with, say, $\Delta -1,$ is
answered by the classical $L^{p}$ regularity theory of elliptic PDEs. In
this case, a necessary and sufficient extra condition is simply $u\in 
\mathcal{S}^{\prime }$ (tempered distributions) since  $\Delta u-u\in L^{p}$
ensures that $u\in W^{2,p}$ (classical Sobolev space). Of course, this is
trivially false for the Laplace operator when $N>1.$

The known sufficient conditions, such as $u\in L^{p}$ (for then $\Delta
u-u\in L^{p}$) or the weaker $(1+|x|^{2})^{-1}u\in L^{p}$ (an implicit
by-product of a result of Nirenberg and Walker in weighted spaces 
\cite[Theorem 3.1]{NiWa73}) or $\nabla u\in (L^{q})^{N}$ for some $%
1<q<\infty $ (Galdi \cite[Remark V.5.3, p. 349]{Ga11}, by duality and
bootstrapping), do not point to any recognizable common feature, especially
since the proof of their sufficiency is each time completely different.

In this paper, we show, among other things, that if $A$ is any \emph{\
homogeneous} elliptic operator of order $m$ with constant complex
coefficients and $Au\in L^{p},$ all the partial derivatives of order $m$ of $%
u$ are in $L^{p}$ if and only if $u$ satisfies a very simple \emph{necessary
and sufficient} side condition. We shall actually prove significantly more
general results in the same spirit. Here and everywhere in the paper,
``homogeneous'' is synonymous with ``pure order'', that is, $A$ is of the
form 
\begin{equation}
Au=i^{m}\sum_{|\alpha |_{1}=m}a_{\alpha }\partial ^{\alpha }u,  \label{1}
\end{equation}
where $m\in \Bbb{N}$ (to avoid trivialities, we rule out $m=0$) and $%
a_{\alpha }\in \Bbb{C}$ and where $|\alpha |_{1}:=\alpha _{1}+\cdots +\alpha
_{N}.$ Recall that the ellipticity assumption means 
\begin{equation}
A(\xi ):=\sum_{|\alpha |_{1}=m}a_{\alpha }\xi ^{\alpha }\neq 0\text{ for
every }\xi \in \Bbb{R}^{N}\backslash \{0\}.  \label{2}
\end{equation}

If $k\in \Bbb{N}_{0}:=\Bbb{N}\cup \{0\}$ and $1\leq p\leq \infty ,$ we
define the homogeneous Sobolev space (also known as Beppo Levi space, after
Deny and Lions \cite{DeLi54}; various other notations, e.g., $L^{k,p},\hat{W}%
^{k,p},BL^{k,p},$ etc., are used in the literature) 
\begin{equation}
D^{k,p}:=\{u\in \mathcal{D}^{\prime }:\partial ^{\alpha }u\in L^{p},|\alpha
|_{1}=k\}=\{u\in L_{loc}^{1}:\partial ^{\alpha }u\in L^{p},|\alpha |_{1}=k\},
\label{3}
\end{equation}
where the second equality follows from the well-known fact that a
distribution with first-order partial derivatives in $L_{loc}^{p}$ is itself
in $L_{loc}^{p}$ (Schwartz \cite[Theorem XV, p. 181]{Sc66}).

When $1<p<\infty ,$ the necessary and sufficient condition for $Au\in L^{p}$
to imply $u\in D^{m,p}$ given in this paper is just a \emph{growth limitation%
} on $|u|$ at infinity, but the correct concept of growth is not a pointwise
one. This is made precise through the introduction of spaces $M^{s,q}$ and
subspaces $M_{0}^{s,q}$ for $s\in \Bbb{R}$ and $1\leq q\leq \infty $
(Section \ref{spaces}). In essence, $u\in M^{s,q}$ ($M_{0}^{s,q}$) if and
only if $|u|$ does not grow faster (grows slower) than $|x|^{s}$ after both
are $L^{q}$-averaged over balls with increasing radii.

On the other hand, since growth slower than $|x|^{0}=1$ must be viewed as
decay, the functions of $M_{0}^{s,q}$ with $s\leq 0,$ all contained in $%
M_{0}^{0,1},$ tend to $0$ at infinity in an $L^{q}$-average sense. These
functions are related to, but have more structure than, the ``functions
vanishing at infinity'' of Lieb and Loss \cite{LiLo01}.

The simplest special case of the main result reads:

\begin{theorem}
\label{thA}If $A$ in (\ref{1}) is elliptic and $u\in \mathcal{D}^{\prime },$
then $u\in D^{m,p}$ for some $1<p<\infty $ if and only if $Au\in L^{p}$ and $%
u\in M_{0}^{m,1}.$ In other words, 
\begin{equation}
D^{m,p}=\{u\in M_{0}^{m,1}:Au\in L^{p}\}.  \label{4}
\end{equation}
\end{theorem}

It is rather remarkable that $u\in D^{m,p}$ -a matter of integrability of $%
\nabla ^{m}u$ at infinity since $Au\in L^{p}$- depends only upon the growth
of $u$ itself and, in addition, that this growth can be evaluated in a $p$%
-independent $L^{1}$ sense. Although the spaces $M^{s,q}$ with $q>1$ are not
involved in this criterion, they are still important for various technical
reasons and in the applications.

In particular, the characterization (\ref{4}) yields an estimate of the
growth at infinity of the functions of the space $D^{m,p}.$ Their \emph{\
pointwise} growth was investigated by Mizuta \cite{Mi86} and, earlier, by
Fefferman \cite{Fe74}, Portnov \cite{Po74}, Uspenski\u {\i } \cite{Us61},
etc. When $mp>N,$ Mizuta's estimates are uniform, but when $mp\leq N$ (so
that functions of $D^{m,p}$ need not be continuous), they are only valid
outside some set thinning out at infinity, which makes them much harder to
use in practice. Uniform pointwise estimates when $m=1$ can also be found in
Galdi's book \cite{Ga11}, but only for functions of $D^{1,p}\cap D^{1,q}$
for some $q>N.$ They coincide with Mizuta's when $p=q>N.$ It has not been
proved, or even suggested, that such pointwise estimates, plus $Au\in L^{p},$
imply $u\in D^{m,p}.$ In other words, there is no prior variant of Theorem 
\ref{thA}.

With a suitable (standard) definition of $D^{k,p}$ when $k<0,$ we actually
prove that, more generally, $u\in D^{m+\kappa ,p}$ for some integer $\kappa
\geq -m$ if and only if $Au\in D^{\kappa ,p}$ and $u\in M_{0}^{m+\kappa ,1}$
(Theorem \ref{th13}). This does not follow inductively from Theorem \ref{thA}%
. When $k<0,$ not only the distributions of $D^{k,p}$ are generally not
functions, but there is no limitation, pointwise or averaged, to the growth
at infinity of the functions of $D^{k,p}$ (Remark \ref{rm4}). For that
reason, there is no predictable generalization of Theorem \ref{thA} when $%
\kappa <-m.$ Also, Theorem \ref{thA} breaks down when $A$ is not
homogeneous. Its validity when $A$ is homogeneous with variable coefficients
is a delicate issue that we shall not address here. It may be false even for
uniformly elliptic operators with smooth, bounded and Lipschitz continuous
coefficients.

The characterization (\ref{4}) calls for a closer look at the functions that
can be found in $M_{0}^{m,1}.$ For the sake of argument, assume $m=2.$ It
can be shown that $u\in M_{0}^{2,1}$ in a variety of special cases,
including:

(i) $u\in L^{q,\sigma }$ or $\nabla u\in (L^{q,\sigma })^{N}$ (Lorentz
spaces, $q>1,\sigma \leq \infty $ or $q=\sigma =1;$ see Example \ref{ex1}
and Theorem \ref{th7}).

(ii) $u\in L_{\Phi }$ or $\nabla u\in (L_{\Phi })^{N}$ where $L_{\Phi }$ is
any Orlicz space; see Example \ref{ex2} and Theorem \ref{th7}.

(iii) $u\in L^{\mathbf{q}}$ or $\nabla u\in (L^{\mathbf{q}})^{N}$ (mixed
norm spaces, $\mathbf{q}=(q_{1},...,q_{N})$ with $1\leq q_{1},...,q_{N}\leq
\infty ;$ see Example \ref{ex3} and Theorem \ref{th7}).

(iv) $u\in W_{loc}^{k,\infty }$ and $\lim_{|x|\rightarrow \infty
}|x|^{-s}|\nabla ^{k}u(x)|=0$ with $0\leq k\leq 2$ and $s\leq 2-k$ (see
Theorem \ref{th3} (ii) and Theorem \ref{th7}).

(v) $(1+|x|)^{-s-N/q}|\nabla ^{k}u|\in L^{q}$ with $1\leq q\leq \infty ,$ $%
0\leq k\leq 2$ and $s<2-k$ (see Theorem \ref{th3}, Remark \ref{rm1} and
Theorem \ref{th7}).

Thus, any of the conditions (i) to (v) together with $\Delta u\in L^{p},$ or
more generally $Au\in L^{p}$ where $A$ is homogeneous second order elliptic
with constant coefficients, implies $u\in D^{2,p}.$ Note that if the
coefficients are not real, $A$ is not reducible to $\Delta $ by a linear
change of variables. The known cases mentioned earlier when $A=\Delta $ are
covered by one or more of these conditions. Evidently, $u\in L^{p}$ fits
within (i), (ii), (iii) and (v), whereas $(1+|x|^{2})^{-1}u\in L^{p}$ is (v)
with $s=2-N/p,k=0$ and $q=p.$ On the other hand, $\nabla u\in (L^{q})^{N}$
with $1<q<\infty $ is also accounted for by (i), (ii), (iii) and (v). Any of
these conditions shows that the values $q=1$ and $q=\infty $ can be
included, even though the argument used in \cite{Ga11} breaks down. In fact,
by (v) with $k=1,$ it suffices that $(1+|x|)^{-t}|\nabla u|\in L^{q}$ with $%
t<1+N/q$ and $1\leq q\leq \infty .$

By (v) with $k=2$ and Theorem \ref{thA}, $(1+|x|)^{-t}|\nabla ^{2}u|\in
L^{q} $ with $1\leq q<\infty $ and $t<N/q$ and $\Delta u\in L^{p}$ with $%
1<p<\infty $ imply $u\in D^{2,p}.$ In particular, if $u\in D^{2,q}$ for some 
$1\leq q<\infty $ and $\Delta u\in L^{p},$ then $u\in D^{2,p}$ (let $t=0;$
if $q=\infty ,$ quadratic harmonic polynomials are counter examples). Also,
if $(1+|x|)^{-t}|\nabla ^{2}u|\in L^{p}$ with $t<N/p$ (weaker than $u\in
D^{2,p}$ if $t>0$) and $\Delta u\in L^{p},$ then $u\in D^{2,p}.$

We now come to the organization of the paper. Section \ref{spaces} is
devoted to the definition and basic properties of the spaces $M^{s,q}.$ With
occasional minor modifications, the growth limitations embodied in these
spaces have already been used extensively when $q=\infty $ or $q=2$ and, in
some instances, when $q=1,$ in connection with Liouville-type theorems (\cite
{Ar01}, \cite{AvLi89}, \cite{LiWa01}, \cite{Lin96}, \cite{MoSt92}) but not
in regularity issues. There seems to have been no prior incentive to
incorporate these growth limitations into a family of function spaces and
the other values of $q$ have apparently been ignored.

The most important feature of the spaces $M^{s,q}$ is that ``integration'',
i.e., passing from $\nabla u$ to $u,$ takes $(M^{s,q})^{N}$ into $M^{s+1,q}$
when $s>-1.$ This is shown in Section \ref{embedding} (Theorem \ref{th7}),
where we also obtain the embedding of $D^{k,p}$ into $M^{s,p}$ for suitable $%
s$ as a straightforward by-product (Theorem \ref{th8}).

This embedding is one of the main ingredients for the proof of Theorem \ref
{th13} (the general form of Theorem \ref{thA}), given in Section \ref
{regularity}. This proof also depends on properties of homogeneous elliptic
operators acting on homogeneous Sobolev spaces (Theorem \ref{th12}). In
addition to new $W^{k,p}$ regularity results which, in particular, do not
require strong ellipticity (Corollary \ref{cor14}), four examples show how
Theorem \ref{th13} and the various properties of the spaces $M^{s,q}$ can be
used in practice.

In Section \ref{exterior}, we generalize Theorem \ref{th13} to exterior
domains (Theorem \ref{th18}). When homogeneous Sobolev spaces of negative
order are involved, this is not a routine variant because passing to an
exterior domain introduces \emph{necessary }restrictions on $N$ and $p,$ not
needed in the whole space.

We also take advantage of the exterior domain setting to show how the Kelvin
transform method yields solutions of boundary value problems in $M_{0}^{0,1}$
(Theorem \ref{th20}). Thus, as noted earlier, these solutions vanish at
infinity in a generalized sense. The physical relevance of solutions
vanishing at infinity has been discussed at length in the literature,
notably in Dautray and Lions \cite{DaLi90}. For obvious reasons, $%
M_{0}^{0,1} $ -larger than any $L^{p}$ space with $p<\infty $- has not
previously been part of this discussion.

In Section \ref{systems}, Theorem \ref{th13} is extended to
Douglis-Nirenberg elliptic systems with constant coefficients when the
entries are homogeneous operators with possibly different orders (Theorem 
\ref{th22}), as is the case with the Stokes system. A variant of a trick
used long ago by Malgrange \cite{Ma56} for other purposes allows for a
convenient reduction to the scalar case.

\textbf{Notation.} The general notation is standard. Everywhere, $B_{R}$ is
the euclidean open ball with center $0$ and radius $R>0$ in $\Bbb{R}^{N}$
and, depending on context, $|\cdot |$ is either the euclidean norm or the
Lebesgue measure. On the other hand, $|\cdot |_{1}$ is the $\ell ^{1}$ norm
(used only with multi-indices). The notation $||\cdot ||_{p,E},$ abbreviated 
$||\cdot ||_{p}$ when $E=\Bbb{R}^{N},$ is used for the norm of $L^{p}(E).$
Also, $p^{\prime }$ denotes the H\"{o}lder conjugate of $p\in [1,\infty ].$

Differentiation is always understood in the weak (distribution) sense and $%
\nabla ^{k}u$ is the symmetric tensor of partial derivatives of order $k\in 
\Bbb{N}$ of the distribution $u.$ Of course, $\nabla u$ is used instead of $%
\nabla ^{1}u.$ As is customary, $\mathcal{S}$ and $\mathcal{S}^{\prime }$
refer to the Schwartz space and its topological dual (tempered
distributions), respectively. Fourier transform on those spaces is denoted
by $\mathcal{F},$ with inverse $\mathcal{F}^{-1}.$ We shall also use the
convenient ``hat'' notation $\widehat{u}:=\mathcal{F}u.$

If $d\in \Bbb{N}_{0},$ we let $\mathcal{P}_{d}$ denote the space of
polynomials on $\Bbb{R}^{N}$ of degree at most $d$ with complex coefficients
and $[u]_{d}$ is the equivalence class of the function $u$ modulo $\mathcal{P%
}_{d}.$ It will be convenient to set $\mathcal{P}_{d}:=\{0\}$ if $d<0$ and $%
\mathcal{P}:=\cup _{d}\mathcal{P}_{d}.$ Lastly, if $X$ and $Y$ are
topological spaces, $X\hookrightarrow Y$ means that $X$ is continuously
embedded into $Y.$

In inequalities, $C>0$ denotes a constant independent of the functions
involved, whose value may change from place to place.

\section{The spaces $M^{s,q}$\label{spaces}}

Unless stated otherwise, $s\in \Bbb{R}$ and $1\leq q\leq \infty .$ We define 
\begin{equation}
M^{s,q}:=\{u\in L_{loc}^{q}:\sup_{R\geq
1}R^{-s}|B_{R}|^{-1/q}||u||_{q,B_{R}}<\infty \}  \label{6}
\end{equation}
and 
\begin{equation}
M_{0}^{s,q}:=\{u\in L_{loc}^{q}:\lim_{R\rightarrow \infty
}R^{-s}|B_{R}|^{-1/q}||u||_{q,B_{R}}=0\},  \label{7}
\end{equation}
where $|B_{R}|^{-1/q}:=1$ if $q=\infty .$ Obviously, $R^{-s-N/q}$ may -and
often will- be substituted for $R^{-s}|B_{R}|^{-1/q}$ in (\ref{6}) and (\ref
{7}). To reconcile these definitions with the comments in the Introduction,
observe that $R^{s+N/q}$ is proportional to $||\,|x|^{s}||_{q,B_{R}}$ when $%
s>-N/q.$

The possible resemblance with Morrey spaces, maximal functions, etc., is
formal at best. In (\ref{6}), the center $0$ of the balls $B_{R}$ is fixed
and the supremum is not taken over all radii $R>0.$ However, there is no
difficulty in showing that as long as $R_{0}>0$ and $x_{0}\in \Bbb{R}^{N}$
are fixed, the definition of $M^{s,q}$ is unchanged if the condition $R\geq
1 $ is replaced with $R\geq R_{0}$ and if all the balls are centered at $%
x_{0}. $

Notice that $M^{s,q}=\{0\}$ if $s<-N/q$ and $M_{0}^{s,q}=\{0\}$ if $s\leq
-N/q.$ Accordingly, all the results quoted without limitation about $s$ are
trivial in these cases. It is equally obvious that $M^{-N/q,q}=L^{q}$ and
that $M^{s_{1},q}\subset M^{s_{2},q}$ and $M_{0}^{s_{1},q}\subset
M_{0}^{s_{2},q}$if $s_{1}\leq s_{2},$ whereas $M^{s_{1},q}\subset
M_{0}^{s_{2},q}$ if $s_{1}<s_{2}$. Also, by H\"{o}lder's inequality, $%
M^{s,q_{2}}\subset M^{s,q_{1}}$ and $M_{0}^{s,q_{2}}\subset M_{0}^{s,q_{1}}$
if $q_{1}\leq q_{2}.$

To summarize, given $q,$ the nontrivial spaces $M^{s,q}$ start with $L^{q}$
when $s=-N/q$ and get larger as $s$ is increased. On the other hand, given $%
s,$ they get smaller as $q$ is increased. The practical value of this double
linear ordering cannot be overemphasized. The spaces $M_{0}^{s,q}$ have
similar properties, except that, given $q,$ there is no smallest nontrivial
space $M_{0}^{s,q}.$

Many classical function spaces are subspaces of some $M^{s,q}$ space.
Examples follow.

\begin{example}
\label{ex1} The Lorentz space $L^{q,\sigma },1<q<\infty ,1\leq \sigma \leq 
\infty ,$ is contained in $M^{-N/q,1}\subset M_{0}^{s,1}$ if $s>-N/q.$ Since 
$L^{q,\sigma }\subset L^{q,\infty },$ it suffices to prove that $L^{q,\infty %
}\subset M^{-N/q,1}.$ The norm of $u\in L^{q,\infty }$ is $||u||_{(q,\infty %
)}:=\sup_{t>0}t^{-1+1/q}\int_{0}^{t}u^{*}(\tau )d\tau <\infty ,$ where $u^{*}
$ is the decreasing rearrangement of $u.$ On the other hand, if $E$ is a
measurable subset of finite measure, then $|E|^{-1}\int_{E}|u|\leq %
|E|^{-1}\int_{0}^{|E|}u^{*}(\tau )d\tau $ (\cite[Theorem 7.3.1, p. 82]{Ga07}%
, \cite[Lemma 3.17, p. 201]{StWe71}). Thus, $|E|^{-1}\int_{E}|u|\leq %
|E|^{-1/q}||u||_{(q,\infty )}.$ By using this with $E=B_{R},$ it follows
that $u\in M^{-N/q,1}$ and that the embedding $L^{q,\sigma }\subset
M^{-N/q,1}$ is continuous. On the other hand, $L^{1,\sigma }\not\subset L_{loc}^{1}$ 
if $\sigma >1$ is not a subspace of any $M^{s,q}.$
\end{example}

\begin{example}
\label{ex2} Let $L_{\Phi }$ be the Orlicz space corresponding to the Young
function $\Phi $ (\cite{BeSh88}, \cite{ON65}). Given $1\leq q<\infty ,$
assume that $t^{q}\leq \Phi (\lambda t)$ if $t\geq t_{0}$ for some $\lambda
>0$ and $t_{0}\geq 0$ ($\lambda $ and $t_{0}$ always exist if $q=1;$ just
choose $t_{0}>0,$ pick $\lambda $ large enough that $\Phi (\lambda t_{0})%
\geq t_{0}$ and use the convexity of $\Phi $). If $v$ is Lebesgue
measurable, then $\int_{B_{R}}|v|^{q}\leq t_{0}^{q}|B_{R}|+\int_{\Bbb{R}%
^{N}}\Phi (\lambda |v|)$ for every $R>0.$ In particular, if $u\in L_{\Phi
}\backslash \{0\},$ the choice $v:=u/\lambda ||u||_{\Phi }$ (Luxemburg norm)
yields $|B_{R}|^{-1}\int_{B_{R}}|u|^{q}\leq \lambda ^{q}||u||_{\Phi
}^{q}(t_{0}^{q}+|B_{R}|^{-1}).$ This shows that (i) $L_{\Phi }\subset
M^{0,q}\subset M_{0}^{s,q}$ for every $s>0$ (true for every $\Phi $ if $q=1$%
), (ii) $L_{\Phi }\subset M^{-N/q,q}=L^{q}$ if $t_{0}=0$ and (iii) $L_{\Phi
}\subset M_{0}^{0,q}$ if $\lambda t_{0}$ can be chosen arbitrarily small. In
particular: (iv) $L_{\Phi }\subset M_{0}^{0,1}$ if and only if $1\notin
L_{\Phi }.$ The necessity is obvious. Conversely, if $1\notin L_{\Phi },$
then $\Phi >0$ on $(0,\infty ),$ so that $\Phi $ has a continuous inverse $%
\Phi ^{-1}$ defined on some interval $[0,b)$ with $0<b\leq \infty .$ For $%
t_{0}<b,$ let $\lambda :=\Phi ^{-1}(t_{0})/t_{0}.$ Then, $t\leq \Phi
(\lambda t)$ if $t\geq t_{0}$ by the monotonicity of $\Phi (\lambda t)/t$
and $\Phi (\lambda t_{0})/t_{0}=1.$ Since $\lambda t_{0}=\Phi
^{-1}(t_{0})\rightarrow 0$ as $t_{0}\rightarrow 0,$ the result follows from
(iii).
\end{example}

In Example \ref{ex2}, $L_{\Phi }\subset M^{0,1}$ for every Orlicz space $%
L_{\Phi }$ can be quickly seen from $L_{\Phi }\subset L^{1}+L^{\infty }$ and 
$L^{1}=M^{-N,1}\subset M^{0,1},L^{\infty }=M^{0,\infty }\subset M^{0,1}.$

\begin{example}
\label{ex3}If $\mathbf{q}:=(q_{1},...,q_{N})$ with $1\leq q_{1},...,q_{N}%
\leq \infty ,$ the space $L^{\mathbf{q}}$ (see \cite{BeIlNi78}) is contained
in $M^{s,q}$ with $s:=-\sum_{i=1}^{N}q_{i}^{-1}$ and $q:=\min
\{q_{1},...,q_{N}\}.$ This follows from the remark that the definition of $%
M^{s,q}$ is unchanged if balls are replaced with cubes.
\end{example}

It is readily checked that 
\begin{equation}
||u||_{M^{s,q}}:=\sup_{R\geq 1}R^{-s-N/q}||u||_{q,B_{R}},  \label{8}
\end{equation}
is a well defined norm on $M^{s,q}.$ The proofs of the first two theorems
are routine and left to the reader (parts (ii) and (iv) of Theorem \ref{th1}
were noticed earlier).

\begin{theorem}
\label{th1}(i) $M^{s,q}$ is a Banach space for the norm (\ref{8}).\newline
(ii) $M^{s_{1},q}\hookrightarrow M^{s_{2},q}$ if $s_{1}\leq s_{2}$ and $%
M^{s,q_{2}}\hookrightarrow M^{s,q_{1}},$ $M_{0}^{s,q_{2}}\hookrightarrow
M_{0}^{s,q_{1}}$ if $q_{1}\leq q_{2}.$\newline
(iii) $M_{0}^{s,q}$ is a closed subspace of $M^{s,q}.$\newline
(iv) $M^{s_{1},q}\hookrightarrow M_{0}^{s_{2},q}$ if $s_{1}<s_{2}.$
\end{theorem}

\begin{theorem}
\label{th2}If $s_{1},s_{2}\in \Bbb{R}$ and if $1\leq q_{1},q_{2}\leq \infty $
satisfy $1/q_{1}+1/q_{2}\leq 1,$ define $s_{3}\in \Bbb{R}$ and $q_{3}\geq 1$
by $s_{3}:=s_{1}+s_{2}$ and $1/q_{3}:=1/q_{1}+1/q_{2}.$ Then, the
multiplication $(u,v)\mapsto uv$ is defined and continuous from $%
M^{s_{1},q_{1}}\times M^{s_{2},q_{2}}$ to $M^{s_{3},q_{3}}$ and from $%
M_{0}^{s_{1},q_{1}}\times M^{s_{2},q_{2}}$ (or $M^{s_{1},q_{1}}\times %
M_{0}^{s_{2},q_{2}}$) to $M_{0}^{s_{3},q_{3}}.$ More precisely, 
\begin{equation*}
||uv||_{M^{s_{3},q_{3}}}\leq ||u||_{M^{s_{1},q_{1}}}||v||_{M^{s_{2},q_{2}}}.
\end{equation*}
\end{theorem}

The above inequality generalizes H\"{o}lder's inequality, which is recovered
when $q_{1}=q,q_{2}=q^{\prime },$ $s_{1}=-N/q$ and $s_{2}=-N/q^{\prime }.$

\begin{example}
\label{ex4}If $a\geq 0,$ the function $(1+|x|)^{a}$ is in $M^{a,\infty }.$
Thus, if $(1+|x|)^{-a}u\in L^{q}=M^{-N/q,q},$ then $u\in M^{a-N/q,q}$ and $%
||u||_{M^{a-N/q,q}}\leq ||((1+|x|)^{a}||_{M^{a,\infty }}||(1+|x|)^{-a}u||_{q}
$ by Theorem \ref{th2}.
\end{example}

The definition of the spaces $M^{s,q}$ hints that they should be related to
weighted Lebesgue spaces with weights behaving like $|x|^{-sq-N}$ for large $%
|x|.$ To make this connection precise, we introduce the spaces 
\begin{equation*}
L_{s}^{q}:=\{u:(1+|x|)^{-s-N/q}u\in L^{q}\},
\end{equation*}
equipped with the obvious norm. This definition makes sense if $s\in \Bbb{R}$
and $1\leq q\leq \infty $ and $L_{s}^{q}=L^{q}(\Bbb{R}^{N};(1+|x|)^{-sq-N}dx)
$ when $q<\infty .$ The only motivation for introducing the spaces $L_{s}^{q}
$ is the proof of Theorem \ref{th4} (i) later. They will not be used beyond
that point, but they play a key role in other issues (\cite{Ra16}).

\begin{theorem}
\label{th3}(i) $M^{s,\infty }=L_{s}^{\infty }$ for every $s\geq 0,$ with
equivalent norms. \newline
(ii) If $s>0,$ then $u\in M_{0}^{s,\infty }$ if and only if $u\in L_{loc}^{%
\infty }$ and, for every $\varepsilon >0,$ there is $R_{\varepsilon }>0$
such that $|u(x)|<\varepsilon |x|^{s}$ for a.e. $x$ with $|x|>R_{\varepsilon
}$ (i.e., $|u(x)|=o(|x|^{s})$ at infinity after modifying $u$ on a null set).%
\newline
(iii) If $1\leq q<\infty ,$ then $L_{s}^{q}\hookrightarrow
M^{s,q}\hookrightarrow L_{t}^{q}$ for every $t>s\geq -N/q$ and $%
L_{s}^{q}\hookrightarrow M_{0}^{s,q}$ if $s>-N/q.$
\end{theorem}

\begin{proof}
(i) $L_{s}^{\infty }\hookrightarrow M^{s,\infty }$ by Example \ref{ex4} with 
$a=s\geq 0$ and $q=\infty .$ To prove $M^{s,\infty }=L_{s}^{\infty }$ with
equivalent norms, it suffices to show that $M^{s,\infty }\subset
L_{s}^{\infty },$ for then the equivalence of norms follows from the inverse
mapping theorem since both $M^{s,\infty }$ and $L_{s}^{\infty }$ are Banach
spaces.

Let then $u\in M^{s,\infty }$ be given and let $n\in \Bbb{N}.$ For a.e. $%
x\in B_{n+1}\backslash B_{n},$ we have $|u(x)|\leq ||u||_{\infty
,B_{n+1}}\leq C(n+1)^{s}$ where $C>0$ is independent of $x$ and $n$ and so $%
(1+|x|)^{-s}|u(x)|\leq |x|^{-s}|u(x)|\leq Cn^{-s}(n+1)^{s}\leq 2^{s}C.$
Thus, $(1+|x|)^{-s}|u(x)|\leq 2^{s}C$ for a.e. $x\in \Bbb{R}^{N}\backslash
B_{1}.$ Since $u\in L_{loc}^{\infty },$ it follows that $(1+|x|)^{-s}u\in
L^{\infty },$ i.e., $u\in L_{s}^{\infty }.$

(ii) The sufficiency is straightforward: Given $\varepsilon >0,$ let $%
R_{\varepsilon }>0$ be such that $|u(x)|<\varepsilon |x|^{s}$ for a.e. $x$
with $|x|>R_{\varepsilon }.$ If $R>R_{\varepsilon },$ then $||u||_{\infty
,B_{R}}<||u||_{\infty ,B_{R_{\varepsilon }}}+\varepsilon R^{s},$ so that $%
R^{-s}||u||_{\infty ,B_{R}}<2\varepsilon $ if $R$ is large enough since $%
s>0. $

Conversely, if $u\in M_{0}^{s,\infty }$ and $\varepsilon >0,$ then $%
||u||_{\infty ,B_{R}}\leq 2^{-s}\varepsilon R^{s}$ for $R$ large enough. In
particular, $|u(x)|\leq 2^{-s}\varepsilon (n+1)^{s}$ for a.e. $x\in
B_{n+1}\backslash B_{n}$ and $n\in \Bbb{N}$ large enough, say $n\geq
n_{\varepsilon },$ and then $|x|^{-s}|u(x)|\leq 2^{-s}\varepsilon
n^{-s}(n+1)^{s}\leq \varepsilon .$ Thus, $|u(x)|\leq \varepsilon |x|^{s}$
for a.e. $x$ with $|x|>R_{\varepsilon }:=n_{\varepsilon }.$

(iii) $L_{s}^{q}\hookrightarrow M^{s,q}$ by Example \ref{ex4} with $%
a=s+N/q\geq 0.$ The proof that $M^{s,q}\hookrightarrow L_{t}^{q}$ when $%
t>s\geq -N/q $ is more delicate. Let $u\in M^{s,q}$ be given. By using $%
B_{n}=\cup _{j=1}^{n}(B_{j}\backslash B_{j-1})$ (with $B_{0}=\emptyset $)
and since $(1+|x|)^{-tq-N}\leq j^{-tq-N}$ for $x\in B_{j}\backslash B_{j-1},$
we get 
\begin{equation*}
\left. 
\begin{array}{r}
\int_{B_{n}}(1+|x|)^{-tq-N}|u|^{q}\leq
\sum_{j=1}^{n}j^{-tq-N}\int_{B_{j}\backslash B_{j-1}}|u|^{q}\qquad \qquad
\qquad \qquad \\ 
=n^{-tq-N}\int_{B_{n}}|u|^{q}+\sum_{j=2}^{n}\left(
(j-1)^{-tq-N}-j^{-tq-N}\right) \int_{B_{j-1}}|u|^{q}.
\end{array}
\right.
\end{equation*}
In the right-hand side, $n^{-tq-N}\int_{B_{n}}|u|^{q}\leq
n^{-(t-s)q}||u||_{M^{s,q}}^{q}$ tends to $0$ as $n\rightarrow \infty $ since 
$t>s.$ Thus, to prove that $u\in L_{t}^{q},$ it suffices to show that the
sum $\sum_{j=2}^{n}\left( (j-1)^{-tq-N}-j^{-tq-N}\right)
\int_{B_{j-1}}|u|^{q}$ is uniformly bounded.

Note that $(j-1)^{-tq-N}-j^{-tq-N}=(j-1)^{-tq-N}(1-(1-j^{-1})^{tq+N})$ and
that $1-(1-j^{-1})^{tq+N}=O(j^{-1})=O((j-1)^{-1})$ for $j\geq 2.$ Thus, $%
(j-1)^{-tq-N}-j^{-tq-N}\leq C(j-1)^{-tq-N-1}$ where $C>0$ is independent of $%
n$ (and of $u$) and so, 
\begin{equation*}
\left. 
\begin{array}{r}
0\leq \sum_{j=2}^{n}\left( (j-1)^{-tq-N}-j^{-tq-N}\right)
\int_{B_{j-1}}|u|^{q}\qquad \qquad \qquad \qquad \qquad \qquad \qquad \qquad
\\ 
\leq
C\sum_{j=2}^{n}(j-1)^{-tq-N-1}\int_{B_{j-1}}|u|^{q}=C%
\sum_{j=1}^{n-1}j^{-(t-s)q-1}j^{-sq-N}\int_{B_{j}}|u|^{q}\qquad \qquad \\ 
\leq C\left( \sum_{j=1}^{n-1}j^{-(t-s)q-1}\right) ||u||_{M^{s,q}}^{q}
\end{array}
\right.
\end{equation*}
and the right-hand side is bounded since $\sum_{j=1}^{\infty
}j^{-(t-s)q-1}<\infty .$ It follows that $u\in L_{t}^{q}$ and, in fact, that 
$||u||_{L_{t}^{q}}^{q}\leq C\left( \sum_{j=1}^{\infty }j^{-(t-s)q-1}\right)
||u||_{M^{s,q}}^{q},$ so that the embedding $M^{s,q}\subset L_{t}^{q}$ is
continuous.

To complete the proof, we now assume $s>-N/q$ and show that $%
L_{s}^{q}\hookrightarrow M_{0}^{s,q}.$ Since $L_{s}^{q}\hookrightarrow
M^{s,q}$ was proved above and $M_{0}^{s,q}\subset M^{s,q},$ it suffices to
show that $L_{s}^{q}\subset M_{0}^{s,q}.$ Let $u\in L_{s}^{q}$ and $%
\varepsilon >0$ be given. Write $u=(1+|x|)^{s+N/q}v$ with $%
v:=(1+|x|)^{-s-N/q}u\in L^{q}$ and choose $\varphi \in \mathcal{C}
_{0}^{\infty }$ such that $||v-\varphi ||_{q}<\varepsilon .$ This yields $%
R^{-s-N/q}||u||_{q,B_{R}}\leq R^{-s-N/q}(1+R)^{s+N/q}\varepsilon
+R^{-s-N/q}||(1+|x|)^{s+N/q}\varphi ||_{q,B_{R}}.$ Now, $||(1+|x|)^{s+N/q}%
\varphi ||_{q,B_{R}}=||(1+|x|)^{s+N/q}\varphi ||_{q,B_{R_{0}}}$ if $R\geq
R_{0}$ and $R_{0}$ is chosen so that $\limfunc{Supp}\varphi \subset
B_{R_{0}}.$ Since $s>-N/q,$ it follows that $\lim \sup_{R\rightarrow \infty
}R^{-s-N/q}||u||_{q,B_{R}}\leq \varepsilon $ and so $u\in M_{0}^{s,q}$ since 
$\varepsilon >0$ is arbitrary.
\end{proof}

\begin{remark}
\label{rm1}If $s<-N/q\leq \widetilde{s},$ then $L_{s}^{q}\hookrightarrow L_{%
\widetilde{s}}^{q}$ and Theorem \ref{th3} is applicable with $s$ replaced
with $\widetilde{s}.$
\end{remark}

Although Theorem \ref{th3} suggests that the gap between the spaces $M^{s,q}$
and $L_{s}^{q}$ is negligible when $s\geq -N/q,$ this gap allows for major
differences. Most notably, the spaces $L_{s}^{q}$ are \emph{not} ordered by
inclusion when $s$ is fixed and $q$ is varied. For example, $%
(1+|x|)^{-N/q^{\prime }}\notin L^{q^{\prime }}$ if $q>1,$ so that there is $%
v\in L^{q}$ such that $(1+|x|)^{-N/q^{\prime }}v\notin L^{1}$ (indeed, if $f$
is a function on $\Bbb{R}^{N}$ and $fv\in L^{1}$ for every $v\in L^{q},$
then $f\in L^{q^{\prime }};$ use truncation and the uniform boundedness
principle.) As a result, $u:=(1+|x|)^{s+N/q}v\in L_{s}^{q}$ but $u\notin
L_{s}^{1}.$ An elaboration on this example shows that $L_{s}^{q_{2}}\not
\subset L_{s}^{q_{1}}$ if $q_{1}<q_{2}$ (or $q_{1}>q_{2},$ which is
trivial). For that reason, it will be technically essential to use the $%
M^{s,q}$ and $M_{0}^{s,q}$ scales than the $L_{s}^{q}$ scale, even though,
by Theorem \ref{th3}, they often end up being interchangeable.

\begin{theorem}
\label{th4} (i) $M^{s,q}\hookrightarrow \mathcal{S}^{\prime }.$ \newline
(ii) If $u$ is a polynomial, then $u\in M^{s,q}$ ($M_{0}^{s,q}$) if and only
if $\deg u\leq s$ ($\deg u<s$).\newline
\end{theorem}

\begin{proof}
(i) With no loss of generality, assume $s\geq -N/q.$ Since $%
M^{s,q}\hookrightarrow M^{s,1}$ by Theorem \ref{th1} (ii), it is not
restrictive to assume $q=1$ and then, by Theorem \ref{th3} (iii), it
suffices to check that $L_{t}^{1}\hookrightarrow \mathcal{S}^{\prime }$ for
every $t\in \Bbb{R}.$ Since $\sup_{x\in \Bbb{R}^{N}}(1+|x|)^{t+N}|\varphi
(x)|$ is a continuous seminorm on $\mathcal{S},$ this follows from 
\begin{equation*}
\left. 
\begin{array}{r}
\left| \int_{\Bbb{R}^{N}}u\varphi \right| \leq
||(1+|x|)^{-t-N}u||_{1}\sup_{x\in \Bbb{R}^{N}}(1+|x|)^{t+N}|\varphi
(x)|\qquad \qquad \qquad \\ 
=||u||_{L_{t}^{1}}\sup_{x\in \Bbb{R}^{N}}(1+|x|)^{t+N}|\varphi (x)|,
\end{array}
\right.
\end{equation*}
for every $u\in L_{t}^{1}$ and every $\varphi \in \mathcal{S}.$

(ii) Let $d:=\deg u.$ It is plain that $u\in M^{d,\infty }\subset
M^{d,q}\subset M^{s,q}$ for every $1\leq q\leq \infty $ and every $s\geq d.$
In particular, $u\in M_{0}^{s,q}$ if $s>d.$

Conversely, assume by contradiction that $u\in M^{s,q}$ for some $1\leq
q\leq \infty $ and some $s<d.$ Choose a system of coordinates such that $%
x=(x_{1},x^{\prime })$ with $x_{1}\in \Bbb{R}^{N},x^{\prime }\in \Bbb{R}
^{N-1}$ and that $u(x)=a_{d}x_{1}^{d}+\sum_{j=0}^{d-1}a_{j}(x^{\prime
})x_{1}^{j},$ where $a_{j}\in \mathcal{P}_{d-j}(\Bbb{R}^{N-1}),0\leq j\leq d$
and $a_{d}\in \Bbb{C}\backslash \{0\}.$

Given $\varepsilon >0,$ denote by $\Sigma _{\varepsilon }\subset \Bbb{R}^{N}$
the sector $|x^{\prime }|<\varepsilon x_{1}$ around the positive $x_{1}$%
-axis. For $0\leq j\leq d-1,$ there are constants $C_{j}\geq 0$ independent
of $\varepsilon $ and of $x\in \Sigma _{\varepsilon }$ such that $%
|a_{j}(x^{\prime })|\leq C_{j}(1+\varepsilon ^{d-j}x_{1}^{d-j})$ for every $%
x\in \Sigma _{\varepsilon }.$ Thus, if $\varepsilon $ is small enough and $%
R_{0}>0$ is large enough, $|u(x)|\geq (|a_{d}|/2)x_{1}^{d}$ for $x\in \Sigma
_{\varepsilon }$ with $|x|\geq R_{0}.$ Since $|x|\leq x_{1}(1+\varepsilon
^{2})^{1/2}$ when $x\in \Sigma _{\varepsilon },$ it follows that $|u(x)|\geq
(|a_{d}|/2)(1+\varepsilon ^{2})^{-d/2}|x|^{d}$ for every $x\in \Sigma
_{\varepsilon }$ with $|x|\geq R_{0}.$ As a result, if $R>R_{0},$%
\begin{equation*}
\left. 
\begin{array}{r}
\int_{\Sigma _{\varepsilon }\cap B_{R}}|u|\geq \frac{|a_{d}|}{2}
(1+\varepsilon ^{2})^{-d/2}\int_{\Sigma _{\varepsilon }\cap (B_{R}\backslash
B_{R_{0}})}|x|^{d}\qquad \qquad \qquad \qquad \\ 
=\frac{\omega _{\varepsilon }|a_{d}|}{2(d+N)}(1+\varepsilon
^{2})^{-d/2}(R^{d+N}-R_{0}^{d+N}),
\end{array}
\right.
\end{equation*}
where $\omega _{\varepsilon }:=|\Sigma _{\varepsilon }\cap B_{1}|/|B_{1}|>0$
is also the ratio of the $N-1$ dimensional measures of $\Sigma _{\varepsilon
}\cap \partial B_{1}$ and $\partial B_{1}.$ Thus, $R^{-s-N}||u||_{1,B_{R}}%
\geq R^{-s-N}\left( \int_{\Sigma _{\varepsilon }\cap B_{R}}|u|\right) \geq
cR^{d-s}\left( 1-R^{-d-N}R_{0}^{d+N}\right) $ where $c=\omega _{\varepsilon
}(|a_{d}|/2)(d+N)^{-1}(1+\varepsilon ^{2})^{-d/2}>0$ is independent of $R.$
Since $d>s,$ this contradicts the assumption $u\in M^{s,q}\subset M^{s,1}.$

If it is assumed that $u\in M_{0}^{s,q}\subset M_{0}^{s,1},$ a contradiction
still arises when $s=d.$ This completes the proof of (ii).
\end{proof}

As a corollary, we obtain an elementary Liouville-type property that will be
instrumental in the proof of Theorem \ref{th13} (and convenient, but not
essential, in that of Lemma \ref{lm10}). Although it will only be used here
with the elliptic operator $A$ in (\ref{1}), we give a more general
statement since the proof is the same. Recall that $\mathcal{P}$ is the
space of polynomials on $\Bbb{R}^{N}$ with complex coefficients.

\begin{corollary}
\label{cor5}Let $B:=\sum_{|\alpha |_{1}\leq m}i^{|\alpha |_{1}}b_{\alpha }%
\partial ^{\alpha }$ be a differential operator with constant coefficients
such that $B(\xi ):=\sum_{|\alpha |_{1}\leq m}b_{\alpha }\xi ^{\alpha }\neq 0
$ on $\Bbb{R}^{N}\backslash \{0\}.$ If $u\in M^{s,q}$ ($M_{0}^{s,q}$) for
some $1\leq q\leq \infty $ and $Bu\in \mathcal{P},$ then $u\in \mathcal{P}$
and $\deg u\leq s$ ($\deg u<s$).
\end{corollary}

\begin{proof}
By Theorem \ref{th4} (i), $u\in \mathcal{S}^{\prime },$ so that $Bu=\pi $
with $\pi \in \mathcal{P}\subset \mathcal{S}^{\prime }$ implies $B(\xi )%
\widehat{u}=\widehat{\pi }.$ Now, $\limfunc{Supp}\widehat{\pi }\subset \{0\}$
since $\widehat{\pi }$ is a linear combination of partial derivatives of $%
\delta $ (Dirac delta). Since $B(\xi )\neq 0$ when $\xi \neq 0,$ it follows
that $\limfunc{Supp}\widehat{u}\subset \{0\}.$ Hence, $\widehat{u}$ is a
linear combination of partial derivatives of $\delta ,$ which amounts to
saying that $u$ is a polynomial. The bound on $\deg u$ follows from Theorem 
\ref{th4} (ii).
\end{proof}

Corollary \ref{cor5} is related in various ways to a number of results in
the literature. Among many others, we mention Agmon, Douglis and Nirenberg 
\cite[p. 662]{AgDoNi59} (when $B$ is homogeneous elliptic and $q=\infty $),
Weck \cite{We83} (when $Bu=0$ and $q=\infty $), H\"{o}rmander \cite{Ho73},
Murata \cite{Mu74} (when $Bu=0,s=0$ and $q=2$). The last two papers deal
with the solutions of $Bu=0$ when $B$ is a general operator with constant
coefficients and $u\in L_{loc}^{2}.$ They are much deeper but only cover the
special case of Corollary \ref{cor5} when $Bu=0,s\leq 0$ and $q\geq 2.$ For
polyharmonic functions, Liouville theorems stronger than Corollary \ref{cor5}
with $q=1,$ but in the same spirit, can be found in Armitage \cite{Ar01} and
the references therein.

\section{Embedding of $D^{k,p}$ into $M^{s,p}$\label{embedding}}

Arguably, the most important feature of the spaces $M^{s,q}$ is that a
function with first order derivatives in $M^{s,q}$ is in $M^{s+1,q}$ if $%
s>-1.$ This will be proved in this section (Theorem \ref{th7}). The examples
given in the Introduction made repeated use of this property. In addition,
the embedding $D^{k,p}\subset M^{s,p}$ for suitable values of $s$ is a
straightforward by-product (Theorem \ref{th8}). For brevity, we do not
discuss the embedding $D^{k,p}\subset M^{s,q}$ when $q\neq p,$ which will
not be needed.

We begin with a lemma on real-valued functions of one variable.

\begin{lemma}
\label{lm6}Let $H:(0,\infty )\rightarrow \Bbb{R}$ be a function bounded
above on every compact subset of $(0,\infty ).$ Suppose that there are $%
\lambda ,\mu \in (0,1)$ and $c\in \Bbb{R},R_{0}>0$ such that 
\begin{equation}
H(R)\leq \mu H(\lambda R)+c,\qquad \forall R\geq R_{0}.  \label{10}
\end{equation}
Then, $H$ is bounded above on $[1,\infty )$ and $\lim \sup_{R\rightarrow 
\infty }H(R)\leq c(1-\mu )^{-1}.$
\end{lemma}

\begin{proof}
Since $H$ is bounded above on the compact subsets of $(0,\infty ),$ it
suffices to show that if $\varepsilon >0,$ then $H(R)\leq \varepsilon
+c(1-\mu )^{-1}$ for $R>0$ large enough. By contradiction, if this is false,
there is a sequence $R_{n}\rightarrow \infty $ such that $\varepsilon
+c(1-\mu )^{-1}<H(R_{n})$ for every $n\in \Bbb{N}.$ With no loss of
generality, assume $R_{n}\geq R_{0}$ and let $j_{n}\in \Bbb{N}$ denote the
largest integer $j$ such that $\lambda ^{j}R_{n}\geq R_{0},$ so that $%
R_{0}\leq \lambda ^{j_{n}}R_{n}<\lambda ^{-1}R_{0}.$ Evidently, $%
j_{n}\rightarrow \infty .$ On the other hand, since $R_{n}\geq R_{0},$ it
follows from (\ref{10}) that $H(R_{n})\leq \mu H(\lambda R_{n})+c$ and so $%
\varepsilon +c(1-\mu )^{-1}<\mu H(\lambda R_{n})+c.$ As a result, $\mu
^{-1}\varepsilon +c(1-\mu )^{-1}<H(\lambda R_{n}).$ If $\lambda R_{n}\geq
R_{0},$ then (\ref{10}) can be used again with $R$ replaced with $\lambda
R_{n},$ which yields $\mu ^{-2}\varepsilon +c(1-\mu )^{-1}<H(\lambda
^{2}R_{n})$ and so, by induction, $\mu ^{-j_{n}}\varepsilon +c(1-\mu
)^{-1}<H(\lambda ^{j_{n}}R_{n}) $ for every $n.$ Since $\lim \mu
^{-j_{n}}=\infty $ and $\lambda ^{j_{n}}R_{n}\in [R_{0},\lambda
^{-1}R_{0})\subset [R_{0},\lambda ^{-1}R_{0}],$ this contradicts the
assumption that $H$ is bounded above on the compact subsets of $(0,\infty ).$
\end{proof}

\begin{theorem}
\label{th7}Let $k\in \Bbb{N}_{0}$ and let $s\in (-1,\infty )$ and $1\leq q%
\leq \infty $ be given. If $u\in \mathcal{D}^{\prime }$ and $\partial %
^{\beta }u\in M^{s,q}$ ($M_{0}^{s,q}$) when $|\beta |_{1}=k,$ then $u\in
M^{s+k,q}$($M_{0}^{s+k,q}$). If, in addition, $k\geq 1,s=0,q=\infty $ and $%
\lim_{|x|\rightarrow \infty }|\nabla ^{k}u(x)|=0,$ then $u\in M_{0}^{k,%
\infty }$ (since $M_{0}^{0,\infty }=\{0\},$ this is false if $k=0$).
\end{theorem}

\begin{proof}
If $k=0,$ there is nothing to prove (and $s>-1$ is not needed). By
induction, it suffices to consider the case when $k=1.$ The same thing is
true for the ``furthermore'' part, for if the result is true when $k=1,$ it
implies that $\partial ^{\beta }u\in M_{0}^{1,\infty }$ when $|\beta
|_{1}=k-1$ and $k>1$ and it suffices to use the first part.

From now on, $k=1.$ We first settle the case $q=\infty .$ Assume that $%
\nabla u\in (M^{s,\infty })^{N}.$ There is a constant $C_{u}>0$ (for
instance, $C_{u}=||\,|\nabla u|\,||_{M^{s,\infty }}$) such that $||\,|\nabla
u|\,||_{\infty ,B_{R}}\leq C_{u}R^{s}$ for every $R\geq 1.$ Thus, $u$ is
Lipschitz continuous with constant $C_{u}R^{s}$ on $B_{R},$ so that $%
|u(x)|\leq |u(0)|+C_{u}R^{s}|x|\leq |u(0)|+C_{u}R^{s+1}$ for $x\in B_{R}.$
As a result, $R^{-s-1}||u||_{\infty ,B_{R}}\leq R^{-s-1}|u(0)|+C_{u}\leq
|u(0)|+C_{u}$ since $s>-1.$ This shows that $u\in M^{s+1,\infty }.$

If now $\nabla u\in (M_{0}^{s,\infty })^{N},$ just note that the constant $%
C_{u}$ above may be chosen arbitrarily small provided that $R$ is large
enough. The result then follows from $\lim \sup_{R\rightarrow \infty
}R^{-s-1}||u||_{\infty ,B_{R}}\leq C_{u}.$

If $\nabla u\in (M^{0,\infty })^{N}=(L^{\infty })^{N}$ and it is only
assumed that $\lim_{|x|\rightarrow \infty }|\nabla u(x)|=0$ (rather than the
trivial $\nabla u\in (M_{0}^{0,\infty })^{N}=\{0\}$), the proof must be
modified to show that $u\in M_{0}^{1,\infty }.$ Note that $u$ is continuous
on $\Bbb{R}^{N}.$

Given $\varepsilon >0,$ there is $R_{0}>0$ such that $||\,|\nabla
u|\,||_{\infty ,\Bbb{R}^{N}\backslash B_{R_{0}}}\leq \varepsilon .$ In
particular, $u $ is Lipschitz continuous with constant $\varepsilon $ in
every ball not intersecting $B_{R_{0}}.$ If $x\notin B_{R_{0}},$ set $%
x_{0}:=R_{0}x/|x|\in \partial B_{R_{0}},$ so that $x_{0}$ is in the closed
ball with center $x$ and radius $|x|-R_{0}$ (not intersecting $B_{R_{0}}$).
As a result, $|u(x)-u(x_{0})|\leq \varepsilon |x-x_{0}|$ and $|x-x_{0}|\leq
|x|$ since $x_{0}$ lies on the line segment between $0$ and $x.$ Hence, $%
|u(x)|\leq |u(x_{0})|+\varepsilon |x|\leq ||u||_{\infty
,B_{R_{0}}}+\varepsilon |x|$ and so $|u(x)|\leq ||u||_{\infty
,B_{R_{0}}}+\varepsilon |x|$ for every $x\in \Bbb{R}^{N}$ since the
inequality is trivial when $x\in B_{R_{0}}.$ This yields $||u||_{\infty
,B_{R}}\leq ||u||_{\infty ,B_{R_{0}}}+\varepsilon R$ for every $R>0,$ whence 
$\lim \sup_{R\rightarrow \infty }R^{-1}||u||_{\infty ,B_{R}}\leq \varepsilon
.$ Thus, $\lim_{R\rightarrow \infty }R^{-1}||u||_{\infty ,B_{R}}=0,$ i.e., $%
u\in M_{0}^{1,\infty }.$

In the remainder of the proof, $q<\infty .$ Let $R>0$ be given. As a first
step, we prove the inequality 
\begin{equation}
\left. 
\begin{array}{c}
||u||_{q,B_{R}}\leq 2\lambda ^{-N/q}||u||_{q,B_{\lambda R}}+2\lambda
^{(1-N)/q}R||\,|\nabla u|\,||_{q,B_{R}},
\end{array}
\right.  \label{11}
\end{equation}
for every $0<\lambda <1$ and every $u\in W^{1,q}(B_{R}).$ Since $q<\infty ,$
it suffices to prove (\ref{11}) when $u\in \mathcal{C}^{\infty }(\overline{B}
_{R}).$ In what follows, $\partial _{\rho }u$ denotes the radial derivative
of $u.$

For $0\leq t<R$ and $\theta \in \Bbb{S}^{N-1},$ write $u(t,\theta
)=u(\lambda t,\theta )+\int_{\lambda t}^{t}\partial _{\rho }u(\tau ,\theta
)d\tau .$ By H\"{o}lder's inequality and since $0<1-\lambda <1,$%
\begin{equation*}
\left. 
\begin{array}{r}
|u(t,\theta )|^{q}\leq 2^{q-1}\left( |u(\lambda t,\theta )|^{q}+\left(
\int_{\lambda t}^{t}|\partial _{\rho }u(\tau ,\theta )|d\tau \right)
^{q}\right) \leq \qquad \qquad \qquad \\ 
2^{q-1}\left( |u(\lambda t,\theta )|^{q}+t^{q-1}\int_{\lambda
t}^{t}|\partial _{\rho }u(\tau ,\theta )|^{q}d\tau \right) .
\end{array}
\right.
\end{equation*}
Multiply by $t^{N-1}$ and use $t\leq \tau /\lambda $ for $\tau \geq \lambda
t $ to get 
\begin{equation*}
\left. 
\begin{array}{r}
t^{N-1}|u(t,\theta )|^{q}\leq 2^{q-1}\lambda ^{1-N}(\lambda
t)^{N-1}|u(\lambda t,\theta )|^{q}+\qquad \qquad \qquad \qquad \\ 
2^{q-1}\lambda ^{1-N}t^{q-1}\int_{\lambda t}^{t}\tau ^{N-1}|\partial _{\rho
}u(\tau ,\theta )|^{q}d\tau .
\end{array}
\right.
\end{equation*}
By integrating over $\Bbb{S}^{N-1}$ and since $\int_{\lambda t}^{t}\tau
^{N-1}|\partial _{\rho }u(\tau ,\theta )|^{q}d\tau \leq \int_{0}^{R}\tau
^{N-1}|\partial _{\rho }u(\tau ,\theta )|^{q}d\tau ,$ we find 
\begin{equation*}
\left. 
\begin{array}{r}
\int_{\Bbb{S}^{N-1}}t^{N-1}|u(t,\theta )|^{q}d\theta \leq \qquad \qquad
\qquad \qquad \qquad \qquad \qquad \qquad \qquad \qquad \qquad \\ 
2^{q-1}\lambda ^{1-N}\int_{\Bbb{S}^{N-1}}(\lambda t)^{N-1}|u(\lambda
t,\theta )|^{q}d\theta +2^{q-1}\lambda ^{1-N}t^{q-1}||\,|\nabla
u|\,||_{q,B_{R}}^{q},
\end{array}
\right.
\end{equation*}
so that (\ref{11}) follows by $t$-integration over $(0,R),$ by using $%
(a+b)^{1/q}\leq a^{1/q}+b^{1/q}$ for $a,b\geq 0,2^{(q-1)/q}<2$ and $%
q^{-1/q}\leq 1.$

Assume now that $\nabla u\in (M^{s,q})^{N},$ so that $u\in W_{loc}^{1,q}.$
In particular, (\ref{11}) holds for every $R>0$ and so 
\begin{multline*}
R^{-s-1-N/q}||u||_{q,B_{R}}\leq \\
2\lambda ^{s+1}(\lambda R)^{-s-1-N/q}||u||_{q,B_{\lambda R}}+2\lambda
^{(1-N)/q}R^{-s-N/q}||\,|\nabla u|\,||_{q,B_{R}},
\end{multline*}
that is, 
\begin{equation}
\left. 
\begin{array}{c}
H(R)\leq 2\lambda ^{s+1}H(\lambda R)+2\lambda
^{(1-N)/q}R^{-s-N/q}||\,|\nabla u|\,||_{q,B_{R}},
\end{array}
\right.  \label{12}
\end{equation}
where $H(R):=R^{-s-1-N/q}||u||_{q,B_{R}}.$ Note that $H$ is continuous on $%
(0,\infty )$ since $q<\infty .$ Choose $\lambda \in (0,1)$ small enough that 
$2\lambda ^{s+1}<1,$ which is possible since $s>-1.$ The assumption $\nabla
u\in (M^{s,q})^{N}$ ensures that $2\lambda ^{(1-N)/q}R^{-s-N/q}||\,|\nabla
u|\,||_{q,B_{R}}$ is bounded irrespective of $R\geq 1.$ By Lemma \ref{lm6}
with $\mu =2\lambda ^{s+1},$ it follows that $H$ is bounded on $[1,\infty )$
and so $u\in M^{s+1,q}.$

If $\nabla u\in (M_{0}^{s,q})^{N},$ then for every $\varepsilon >0$ there is 
$R_{\varepsilon }>0$ such that $R^{-s-N/q}||\,|\nabla
u|\,||_{q,B_{R}}<\varepsilon $ if $R\geq R_{\varepsilon }.$ Thus, by (\ref
{12}), $H(R)\leq 2\lambda ^{s+1}H(\lambda R)+2\lambda ^{(1-N)/q}\varepsilon $
if $R\geq R_{\varepsilon }.$ By Lemma \ref{lm6} (still with $\mu =2\lambda
^{s+1}<1$), $\lim \sup_{R\rightarrow \infty }H(R)\leq 2\lambda
^{(1-N)/q}\varepsilon (1-2\lambda ^{s+1})^{-1},$ so that $\lim_{R\rightarrow 
\infty }H(R)=0$ since $\varepsilon >0$ is arbitrary. Hence, $u\in
M_{0}^{s+1,q}.$
\end{proof}

If $k\geq 1$ and $s<-1,$ Theorem \ref{th7} is always false: If $u$ is a
polynomial of degree exactly $k-1\geq 0,$ then $\nabla ^{k}u=0\in
M_{0}^{s,q} $ but $u\notin M^{s+k,q}$ since $s+k<k-1$ (Theorem \ref{th4}
(ii)). If $s=-1, $ it takes a bit more work to show that it is true from $%
M^{-1,q}$ to $M^{k-1,q}$ (but never from $M_{0}^{-1,q}$ to $M_{0}^{k-1,q}$)
in only two cases with little to no interest: If $q>N$ (trivial since $%
M^{-1,q}=\{0\}$ and $\nabla ^{k}u=0$ if and only if $u\in \mathcal{P}
_{k-1}\subset M^{k-1,q}$) or if $q=N=1$ (if $u^{(k)}\in M^{-1,1}=L^{1},$
then $u^{(k-1)}\in L^{\infty }=M^{0,\infty },$ so that $u\in M^{k-1,\infty
}\subset M^{k-1,1}$ by Theorem \ref{th7} with $s=k-1\geq 0$).

\begin{theorem}
\label{th8}(i) If $k\in \Bbb{N}_{0}$ and $N<p\leq \infty ,$ then $%
D^{k,p}\subset M^{s,p}$ for every $s\geq k-N/p$ and $D^{k,p}\subset
M_{0}^{s,p}$ for every $s>k-N/p.$\newline
(ii) If $k\in \Bbb{N}_{0}$ and $1\leq p\leq N,$ then $D^{k,p}\subset
M_{0}^{s,p}$ for every $s>k-1.$\newline
\end{theorem}

\begin{proof}
If $u\in D^{k,p},$ then $\partial ^{\beta }u\in L^{p}=M^{-N/p,p}$ when $%
|\beta |_{1}=k.$ In case (i), Theorem \ref{th7} with $q=p$ and $s-k$ instead
of $s$ directly yields $u\in M^{s,p}$ for $s\geq k-N/p$ because $s-k\geq
-N/p>-1$ and $M^{-N/p,p}\subset M^{s-k,p}.$ That $D^{k,p}\subset M_{0}^{s,p}$
for every $s>k-N/p$ follows from Theorem \ref{th1} (iv).

In case (ii), $-N/p\leq -1,$ so that $M^{-N/p,p}\subset M^{-1,p}\subset
M_{0}^{-1+\varepsilon ,p}$ for every $\varepsilon >0$ (Theorem \ref{th1}
(iv)) and then $u\in M_{0}^{k-1+\varepsilon ,p}$ by Theorem \ref{th7} with $%
s=$ $-1+\varepsilon .$
\end{proof}

Since $||\,|\nabla ^{k}u|\,||_{p}$ is only a seminorm on $D^{k,p}$ when $%
k\geq 1,$ the embeddings of Theorem \ref{th8} are not continuous if $D^{k,p}$
is equipped with this seminorm, but it is easy to get around this problem.
For $k\in \Bbb{N}_{0},$ we define 
\begin{equation}
\left. 
\begin{array}{c}
\dot{D}^{k,p}:=D^{k,p}/\mathcal{P}_{k-1},
\end{array}
\right.  \label{13}
\end{equation}
a (reflexive if $1<p<\infty $) Banach space for the norm $||[u]_{k-1}||_{%
\dot{D}^{k,p}}:=||\,|\nabla ^{k}u|\,||_{p}$ with $u\in D^{k,p},$ where $%
[u]_{k-1}$ denotes the equivalence class modulo $\mathcal{P}_{k-1}$ (\cite
{Ga11}, \cite[p. 22]{Ma11}). Recall that $\mathcal{P}_{-1}=\{0\},$ so that $%
\dot{D}^{0,p}=D^{0,p}=L^{p}.$

\begin{remark}
\label{rm2}By a theorem of Sobolev \cite{So63}, $\mathcal{C}_{0}^{\infty }$
is dense in $\dot{D}^{k,p}$ if $1<p<\infty .$ See Hajlasz and Kalamajska 
\cite{HaKa95} for a simple proof and for the case $p=1,N>1.$
\end{remark}

The next theorem asserts that the set-theoretic embedding $D^{k,p}\subset
M^{s,q}$ (or $D^{k,p}\subset M_{0}^{s,q}$) is always equivalent to the
topological embedding $\dot{D}^{k,p}\hookrightarrow M^{s,q}/\mathcal{P}
_{k-1} $ (or $\dot{D}^{k,p}\hookrightarrow M_{0}^{s,q}/\mathcal{P}_{k-1}$).
In particular, the embedding of quotient spaces corresponding to the
embeddings of Theorem \ref{th8} are continuous.

\begin{theorem}
\label{th9} If $k\in \Bbb{N},s\geq k-1$ ($s>k-1$) and $1\leq p,q\leq \infty ,
$ then $\dot{D}^{k,p}\hookrightarrow M^{s,q}/\mathcal{P}_{k-1}$ ($\dot{D}%
^{k,p}\hookrightarrow M_{0}^{s,q}/\mathcal{P}_{k-1}$) if and only if $%
D^{k,p}\subset M^{s,q}$ ($D^{k,p}\subset M_{0}^{s,q}$).
\end{theorem}

\begin{proof}
The necessity is obvious. For the sufficiency, we only prove that if $s\geq
k-1$ and $D^{k,p}\subset M^{s,q},$ then $\dot{D}^{k,p}\hookrightarrow
M^{s,q}/\mathcal{P}_{k-1}.$ The exact same argument works in the other case.

Both $\dot{D}^{k,p}$ and $M^{s,q}/\mathcal{P}_{k-1}$ are Banach spaces, the
latter since $M^{s,q}$ is complete (Theorem \ref{th1} (i)) and $\mathcal{P}%
_{k-1}$ is finite dimensional. Therefore, by the closed graph theorem, it
suffices to show that if $[u_{n}]_{k-1}\in \dot{D}^{k,p}$ is a sequence such
that $[u_{n}]_{k-1}\rightarrow [u]_{k-1}$ in $\dot{D}^{k,q}$ and $%
[u_{n}]_{k-1}\rightarrow [v]_{k-1}$ in $M^{s,q}/\mathcal{P}_{k-1},$ then $%
[u]_{k-1}=[v]_{k-1}.$

That $[u_{n}]_{k-1}\rightarrow [u]_{k-1}$ in $\dot{D}^{k,p}$ means that $%
u_{n},u\in D^{k,p}$ and $\partial ^{\alpha }u_{n}\rightarrow \partial
^{\alpha }u$ in $L^{p}$ when $|\alpha |_{1}=k.$ On the other hand, since $%
M^{s,q}$ is a normed space, $[u_{n}]_{k-1}\rightarrow [v]_{k-1}$ in $M^{s,q}/%
\mathcal{P}_{k-1}$ means that $v\in M^{s,q}$ and that there is a sequence $%
\pi _{n}\in \mathcal{P}_{k-1}$ such that $u_{n}-\pi _{n}\rightarrow v$ in $%
M^{s,q}.$

By Theorem \ref{th4} (i), $u_{n}-\pi _{n}\rightarrow v$ in $\mathcal{S}
^{\prime }.$ Since $\deg \pi _{n}\leq k-1$ and differentiation is continuous
on $\mathcal{S}^{\prime },\partial ^{\alpha }u_{n}\rightarrow \partial
^{\alpha }v$ in $\mathcal{S}^{\prime }$ when $|\alpha |_{1}=k.$ As a result, 
$\partial ^{\alpha }u=\partial ^{\alpha }v$ since $\partial ^{\alpha
}u_{n}\rightarrow \partial ^{\alpha }u$ in $L^{p}\hookrightarrow \mathcal{S}
^{\prime }.$ Thus, $u-v\in \mathcal{P}_{k-1},$ i.e., $[u]_{k-1}=[v]_{k-1}.$
\end{proof}

\section{Regularity\label{regularity}}

In this section, we prove a more general form of Theorem \ref{thA}. The
proof will follow from Theorem \ref{th8} combined with another preliminary
result (Theorem \ref{th12} below) with somewhat of a folklore status. For
instance, related inequalities are quickly mentioned in Bergh and
L\"{o}fstr\"{o}m \cite[p. 167]{BeLo76}, with a reference to H\"{o}rmander 
\cite{Ho63}, where apparently nothing relevant is to be found. When $A=%
\Delta ,$ special cases have been proved ``as needed'' (\cite{FaSo94} when $%
\kappa =-1,$ \cite{GaMeSpTa14} when $\kappa =0,$ \cite{SiSo96} when $\kappa 
\geq 0$). For completeness, we give a full self-contained proof. The
following lemma is the first step.

\begin{lemma}
\label{lm10} If $\kappa \in \Bbb{N}_{0}$ and $1<p<\infty ,$ the homogeneous
elliptic operator $A$ in (\ref{1}) is a linear isomorphism from $\dot{D}%
^{m+\kappa ,p}$ onto $\dot{D}^{\kappa ,p}.$
\end{lemma}

\begin{proof}
If $v,w\in \dot{D}^{m+\kappa ,p}$ and $[v]_{m+\kappa -1}=[w]_{m+\kappa -1},$
then $w=v+\pi $ where $\pi \in \mathcal{P}_{m+\kappa -1},$ whence $%
Aw=Av+A\pi $ with $A\pi \in \mathcal{P}_{\kappa -1}$ (since $A$ is
homogeneous of order $m$). Thus, $[Av]_{\kappa -1}$ is independent of the
representative of $[v]_{m+\kappa -1},$ so that $A:$ $\dot{D}^{m+\kappa
,p}\rightarrow $ $\dot{D}^{\kappa ,p}$ is well defined by $A[v]_{m+\kappa
-1}:=[Av]_{\kappa -1}.$ From the definitions of the norms of $\dot{D}
^{m+\kappa ,p}$ and $\dot{D}^{\kappa ,p},$ it follows at once that $A:\dot{D}
^{m+\kappa ,p}\rightarrow \dot{D}^{\kappa ,p}$ is continuous.

If $[v]_{m+\kappa -1}\in \dot{D}^{m+\kappa ,p}$ and $A[v]_{m+\kappa
-1}=[0]_{\kappa -1},$ then $[Av]_{\kappa -1}=[0]_{\kappa -1}$ for every
representative $v$ of $[v]_{m+\kappa -1},$ i.e., $Av\in \mathcal{P}_{\kappa
-1}.$ By Theorem \ref{th8}, $D^{m+\kappa ,p}\subset M_{0}^{m+\kappa ,p}$ and
so, by Corollary \ref{cor5}, $v\in \mathcal{P}_{m+\kappa -1}.$ Thus, $%
[v]_{m+\kappa -1}=[0]_{m+\kappa -1}$ and so $A$ is one-to-one on $\dot{D}%
^{m+\kappa ,p}.$

It remains to prove the surjectivity of $A.$ A distribution $E\in \mathcal{S}%
^{\prime }$ is a fundamental solution of $A$ if and only if $E=\mathcal{F}%
^{-1}\widehat{E},$ where $\widehat{E}\in \mathcal{S}^{\prime }$ is any
distribution such that $A(\cdot )\widehat{E}=1.$ Although the existence of a
tempered fundamental solution has long been known for all nonzero
differential operators with constant coefficients (H\"{o}rmander \cite{Ho58}%
, Lojasiewicz \cite{Lo59}), we need a more precise result in the simpler
case of this lemma. The construction below, needed to clarify an important
point, is implicit in H\"{o}rmander \cite{Ho83}; see also Camus \cite{Ca06}.

If $m<N,$ it follows from the ellipticity and homogeneity of $A$ that the
function $\widehat{E}:=A(\cdot )^{-1}\in L_{loc}^{1}\cap \mathcal{S}^{\prime
}$ solves the division problem. If $m\geq N,$ we may proceed as follows.
Given $\phi (=\phi (\xi ))\in \mathcal{C}_{0}^{\infty }$ and $\rho \geq 0,$
set 
\begin{equation*}
\left. 
\begin{array}{c}
\psi _{\phi }(\rho ):=\int_{\Bbb{S}^{N-1}}A(\sigma )^{-1}\phi (\rho \sigma
)d\sigma .
\end{array}
\right. 
\end{equation*}
This makes sense since $A(\sigma )^{-1}$ is bounded away from $0$ on $\Bbb{S}%
^{N-1}.$ Clearly, $\psi _{\phi }\in \mathcal{C}_{0}^{\infty }([0,\infty ))$
with 
\begin{equation}
\left. 
\begin{array}{c}
\psi _{\phi }^{(j)}(\rho )=\int_{\Bbb{S}^{N-1}}A(\sigma )^{-1}D^{j}\phi
(\rho \sigma )\sigma ^{j}d\sigma ,\qquad \forall j\in \Bbb{N}_{0}.
\end{array}
\right.   \label{14}
\end{equation}

Formally, $\widehat{E}\,``="A(\cdot )^{-1}$ should be given by $\langle 
\widehat{E},\phi \rangle =\int_{0}^{\infty }\rho ^{N-1-m}\psi _{\phi }(\rho
)d\rho $ but, since $\rho ^{N-1-m}\notin L_{loc}^{1}([0,\infty ))$ when $%
m\geq N,$ the integral is not defined. We replace it with its finite part
(see Schwartz \cite[p. 42]{Sc66}, H\"{o}rmander \cite[p. 69 ff]{Ho83}),
thereby defining $\widehat{E}$ by 
\begin{multline}
\langle \widehat{E},\phi \rangle =  \label{15} \\
\left. 
\begin{array}{r}
\lim_{\varepsilon \rightarrow 0^{+}}\left[ \int_{\varepsilon }^{\infty }\rho
^{N-1-m}\psi _{\phi }(\rho )d\rho +\sum_{j=0}^{m-N-1}\frac{\psi _{\phi
}^{(j)}(0)}{j!}\frac{\varepsilon ^{N-m+j}}{N-m+j}+\frac{\psi _{\phi
}^{(m-N)}(0)}{(m-N)!}\log \varepsilon \right] \\ 
=\frac{1}{(m-N)!}\left[ -\int_{0}^{\infty }(\log \rho )\psi _{\phi
}^{(m-N+1)}(\rho )d\rho +\left( \sum_{j=1}^{m-N}j^{-1}\right) \psi _{\phi
}^{(m-N)}(0)\right] .
\end{array}
\right.
\end{multline}
It is readily checked that $\widehat{E}$ is a distribution on $\Bbb{R}^{N}$
and that $\widehat{E}\in \mathcal{S}^{\prime }.$ Furthermore, $\psi
_{A(\cdot )\phi }(\rho )=\rho ^{m}\int_{\Bbb{S}^{N-1}}\phi (\rho \sigma
)d\sigma $ (in particular, $\psi _{A(\cdot )\phi }^{(j)}(0)=0$ if $j\leq m-1$%
), which shows that $\langle \widehat{E},A(\cdot )\phi \rangle =\int_{\Bbb{R}%
^{N}}\phi ,$ i.e., that $A(\cdot )\widehat{E}=1,$ as desired.

Thus, $E:=\mathcal{F}^{-1}\widehat{E}\in \mathcal{S}^{\prime }$ is a
fundamental solution. For our purposes, the key property of $\widehat{E}$ is
that $\xi ^{\alpha }\widehat{E}$ is the bounded \emph{function} $\xi
^{\alpha }A(\xi )^{-1}$ when $|\alpha |_{1}=m.$ This is obvious if $m<N,$
for then $\widehat{E}$ is already a function. If $m\geq N,$ it follows from (%
\ref{14}) that $\psi _{\xi ^{\alpha }\phi }^{(j)}(0)=0$ for every $\phi \in 
\mathcal{C}_{0}^{\infty }$ when $j\leq m-1,$ so that, by (\ref{15}%
),\linebreak $\langle \widehat{E},\xi ^{\alpha }\phi \rangle
=\lim_{\varepsilon \rightarrow 0^{+}}\int_{\varepsilon }^{\infty }\rho
^{N-1-m}\psi _{\xi ^{\alpha }\phi }(\rho )d\rho =\lim_{\varepsilon
\rightarrow 0^{+}}\int_{|\xi |\geq \varepsilon }\xi ^{\alpha }A(\xi
)^{-1}\phi (\xi )d\xi =$\linebreak $\int_{\Bbb{R}^{N}}\xi ^{\alpha }A(\xi
)^{-1}\phi (\xi )d\xi ,$ which proves the claim.

In the remainder of the proof, $\alpha ,\beta $ and $\gamma $ denote
multi-indices and $|\alpha |_{1}=m.$ If $\varphi (=\varphi (x))\in \mathcal{C%
}_{0}^{\infty },$ then $E*\varphi \in \mathcal{S}^{\prime }$ (%
\cite[Theorem XI, p. 247]{Sc66}) solves $A(E*\varphi )=\varphi .$ Also, $%
\partial ^{\alpha +\beta }(E*\varphi )=\partial ^{\alpha }E*$ $\partial
^{\beta }\varphi \in \mathcal{S}^{\prime }$ and so $\mathcal{F(}\partial
^{\alpha +\beta }(E*\varphi ))=\widehat{\partial ^{\alpha }E}\widehat{%
\partial ^{\beta }\varphi }$ (e.g., because $\partial ^{\beta }\varphi \in 
\mathcal{C}_{0}^{\infty }$ and $\partial ^{\alpha }E\in \mathcal{S}^{\prime
};$ see \cite[p. 268]{Sc66}). Since $\widehat{\partial ^{\alpha }E}%
=(-i)^{m}\xi ^{\alpha }\widehat{E}$ and since $\xi ^{\alpha }\widehat{E}$ is
the function $\xi ^{\alpha }A(\xi )^{-1},$ it follows that $\partial
^{\alpha +\beta }(E*\varphi )=(-i)^{m}\mathcal{F}^{-1}(\xi ^{\alpha }A(\xi
)^{-1}\widehat{\partial ^{\beta }\varphi }).$

By the Mikhlin multiplier theorem (\cite[p. 96]{St70}), $\mathcal{F}
^{-1}(\xi ^{\alpha }A(\xi )^{-1}\mathcal{F})$ is a bounded operator on $%
L^{p}.$ As a result, $\partial ^{\alpha +\beta }(E*\varphi )\in L^{p}$ and
there is a constant $C_{\alpha }>0$ independent of $\varphi \in \mathcal{C}
_{0}^{\infty }$ such that $||\partial ^{\alpha +\beta }(E*\varphi
)||_{p}\leq C_{\alpha }||\partial ^{\beta }\varphi ||_{p}.$ Since every $%
\gamma $ with $|\gamma |_{1}=m+\kappa $ can be split in the form $\gamma
=\alpha +\beta $ with $|\alpha |_{1}=m$ and $|\beta |_{1}=\kappa ,$ this
shows that $E*\varphi \in D^{m+\kappa ,p}$ and that $||[E*\varphi
]_{m+\kappa -1}||_{\dot{D}^{m+\kappa ,p}}\leq C||[\varphi ]_{\kappa -1}||_{%
\dot{D}^{\kappa ,p}}$ where $C>0$ is independent of $\varphi .$

As noted in Remark \ref{rm2}, $\mathcal{C}_{0}^{\infty }$ is dense in $\dot{D%
}^{\kappa ,p}.$ Given $f\in D^{\kappa ,p},$ let then $\varphi _{n}\in 
\mathcal{C}_{0}^{\infty }$ tend to $[f]_{\kappa -1}$ in $\dot{D}^{\kappa
,p}, $ i.e., $\partial ^{\beta }\varphi _{n}\rightarrow \partial ^{\beta }f$
in $L^{p}$ when $|\beta |_{1}=\kappa .$ From the above, the sequence $%
[E*\varphi _{n}]_{m+\kappa -1}$ is a Cauchy sequence in $\dot{D}^{m+\kappa
,p},$ so that it has a limit $[v]_{m+\kappa -1}\in \dot{D}^{m+\kappa ,p}.$
By the continuity of $A,$ the convergence of $[E*\varphi _{n}]_{m+\kappa -1}$
to $[v]_{m+\kappa -1}$ in $\dot{D}^{m+\kappa ,p}$ implies $A[v]_{m+\kappa
-1}=\lim A[E*\varphi _{n}]_{m+\kappa -1}=\lim [A(E*\varphi _{n})]_{\kappa
-1}.$ Since $A(E*\varphi _{n})=\varphi _{n},$ this yields $A[v]_{m+\kappa
-1}=\lim [\varphi _{n}]_{\kappa -1}=[f]_{\kappa -1}.$ Thus, $A$ is onto $%
\dot{D}^{\kappa ,p}$ and the proof is complete.
\end{proof}

When $\kappa >0,$ neither Lemma \ref{lm10} nor its proof implies that $A$
maps $D^{m+\kappa ,p}$ onto $D^{\kappa ,p}.$ This issue will be resolved in
Theorem \ref{th12}.

If $\ell \in \Bbb{N}_{0}$ and $1<p<\infty ,$ we now set 
\begin{equation}
\left. 
\begin{array}{c}
D^{-\ell ,p}:=\left( \dot{D}^{\ell ,p^{\prime }}\right) ^{*},
\end{array}
\right.  \label{16}
\end{equation}
a Banach space for the dual norm. Consistent with (\ref{3}), this gives
again $D^{0,p}=L^{p}.$ Denote by 
\begin{equation}
\left. 
\begin{array}{c}
\nu (\ell ,N):=\binom{N+\ell -1}{\ell },
\end{array}
\right.  \label{17}
\end{equation}
the number of multi-indices $\alpha $ such that $|\alpha |_{1}=\ell .$ Since
the mapping $[u]_{\ell -1}\mapsto \nabla ^{\ell }u$ is an isometric
isomorphism of $\dot{D}^{\ell ,p^{\prime }}$ onto a closed subspace of $%
(L^{p^{\prime }})^{\nu (\ell ,N)},$ it follows from the Hahn-Banach theorem
that every $f\in D^{-\ell ,p}$ has the form $\langle f,[u]_{\ell -1}\rangle
=\sum_{|\alpha |_{1}=\ell }\int_{\Bbb{R}^{N}}f_{\alpha }\partial ^{\alpha }u$
where $f_{\alpha }\in L^{p}$ and $||f||_{D^{-\ell ,p}}=||(\sum_{|\alpha
|_{1}=\ell }|f_{\alpha }|^{2})^{1/2}||_{p}.$

Conversely, if $f\in D^{-\ell ,p}$ is defined by $\langle f,[u]_{\ell
-1}\rangle :=\sum_{|\alpha |_{1}=\ell }\int_{\Bbb{R}^{N}}f_{\alpha }\partial
^{\alpha }u$ for some $f_{\alpha }\in L^{p},$ then $||f||_{D^{-\ell ,p}}\leq
||(\sum_{|\alpha |_{1}=\ell }|f_{\alpha }|^{2})^{1/2}||_{p}$ and, by the
denseness of $\mathcal{C}_{0}^{\infty }$ in $\dot{D}^{\ell ,p^{\prime }}$
(Remark \ref{rm2}), $f\in D^{-\ell ,p}$ is uniquely determined by the
distribution $\sum_{|\alpha |_{1}=1}(-1)^{\ell }\partial ^{\alpha }f_{\alpha
}.$ By changing $f_{\alpha }$ into $(-1)^{\ell }f_{\alpha },$ it follows
that 
\begin{equation}
\left. 
\begin{array}{c}
D^{-\ell ,p}=\{f=\sum_{|\alpha |_{1}=\ell }\partial ^{\alpha }f_{\alpha
}:f_{\alpha }\in L^{p}\},
\end{array}
\right.  \label{18}
\end{equation}
equipped with the norm $\inf ||(\sum_{|\alpha |_{1}=\ell }|f_{\alpha
}|^{2})^{1/2}||_{p}$ (always a minimum). In particular, this shows that $%
\partial ^{\beta }$ maps $D^{\kappa ,p}$ into $D^{\kappa -|\beta |_{1},p}$
for every $\kappa \in \Bbb{Z}.$

We shall now extend Lemma \ref{lm10} when $\kappa \in \Bbb{Z}.$ To do this,
we need another lemma, in the spirit of Corollary \ref{cor5}.

\begin{lemma}
\label{lm11}If $1<p<\infty $ and $k\in \Bbb{Z},$ then $D^{k,p}\cap \mathcal{P%
}=\mathcal{P}_{k-1}.$
\end{lemma}

\begin{proof}
If $k\geq 0,$ the result is trivial since $L^{p}$ contains no nonzero
polynomial. If $k<0,$ it must be shown that if $u\in D^{k,p}$ is a
polynomial, then $u=0.$

Set $k=-\ell $ with $\ell \in \Bbb{N}.$ We first prove that $u$ cannot be a
nonzero constant. By contradiction, if $1\in D^{-\ell ,p},$ it follows from (%
\ref{18}) that $1=\sum_{|\alpha |_{1}=\ell }\partial ^{\alpha }f_{\alpha }$
for some $f_{\alpha }\in L^{p}$ and so $\left| \int_{\Bbb{R}^{N}}\varphi
\right| \leq C||\,|\nabla ^{\ell }\varphi |\,||_{p^{\prime }}$ for every $%
\varphi \in \mathcal{C}_{0}^{\infty },$ where $C>0$ depends only upon the
norms $||f_{\alpha }||_{p}.$ Let $\psi \in \mathcal{C}_{0}^{\infty }$ be
such that $\int_{\Bbb{R}^{N}}\psi =1.$ With $\varphi (x):=\psi (\lambda x)$
and $\lambda >0,$ we get $1=\left| \int_{\Bbb{R}^{N}}\psi \right| \leq
C\lambda ^{\ell +N/p}||\,|\nabla ^{\ell }\psi |\,||_{p^{\prime }}$ with the
same constant $C$ independent of $\lambda $ and a contradiction arises by
letting $\lambda \rightarrow 0.$

If now $u\in D^{-\ell ,p}$ is a nonzero polynomial, some derivative of $u$
is a nonzero constant and this derivative is in $D^{-\widetilde{\ell },p}$
with $\widetilde{\ell }\geq \ell >0,$ which contradicts $1\notin D^{-%
\widetilde{\ell },p}.$
\end{proof}

\begin{remark}
\label{rm3}The exact same line of argument can be used to show that $%
W^{k,p},k\in \Bbb{Z},$ contains no nonzero polynomial.
\end{remark}

\begin{remark}
\label{rm4}Lemma \ref{lm11} may suggest that when $k<0,$ the \emph{functions}
of $D^{k,p}$ continue to be subject to growth limitations at infinity. This
is false. For example, $g_{n}(x):=(1+|x|^{2})^{-N/2}e^{i|x|^{2n}}$ is in $%
L^{p}$ for every $1<p<\infty $ and every $n\in \Bbb{N},$ whence $f_{n}:=%
\partial _{1}g_{n}\in D^{-1,p}$ by (\ref{18}), but $f_{n}\notin M^{s,q}$ for
any prescribed $s$ and $q$ if $n$ is large enough. Nonetheless, depending on 
$N,$ $p$ and $k,$ suitable integrability conditions suffice for membership
to $D^{k,p}$ when $k<0$ (Lemma \ref{lm15}), but this is not always true
(Example \ref{ex8}).
\end{remark}

To give uniform statements for all $k\in \Bbb{Z},$ we henceforth drop the
``dot'' notation $\dot{D}^{k,p}$ when $k\geq 0$ and return to the usual
quotient space notation. Of course, $D^{k,p}/\mathcal{P}_{k-1}=D^{k,p}$ when 
$k\leq 0.$

\begin{theorem}
\label{th12}If $\kappa \in \Bbb{Z}$ and $1<p<\infty ,$ the homogeneous
elliptic operator $A$ in (\ref{1}) is a linear isomorphism from $D^{m+\kappa
,p}/\mathcal{P}_{m+\kappa -1}$ onto $D^{\kappa ,p}/\mathcal{P}_{\kappa -1}$
and a homomorphism of $D^{m+\kappa ,p}$ onto $D^{\kappa ,p}.$ \newline
\end{theorem}

\begin{proof}
We begin with the isomorphism property. Since it was proved in Lemma \ref
{lm10} when $\kappa \geq 0,$ we assume $\kappa <0.$ Note first that $%
D^{k,p}\subset \mathcal{S}^{\prime }$ for every $k\in \Bbb{Z}.$ This follows
for instance from Theorem \ref{th4} (i) and Theorem \ref{th8} if $k\geq 0$
(alternatively, \cite[pp. 244-245]{Sc66} shows that $u\in \mathcal{S}%
^{\prime }$ if and only if all the derivatives of $u$ of some order $k\geq 0$
are in $\mathcal{S}^{\prime }$) and from (\ref{18}) if $k<0.$

By the homogeneity and ellipticity of $A,$ the only solutions $u\in \mathcal{%
S}^{\prime }$ of $Au=0$ are polynomials. This is a simple exercise on
Fourier transform (see the proof of Corollary \ref{cor5}). Consequently, if $%
u\in D^{m+\kappa ,p}$ and $Au=0,$ then $u\in \mathcal{P}$ and so $u\in 
\mathcal{P}_{m+\kappa -1}$ by Lemma \ref{lm11}. Thus, $A$ is one-to-one on $%
D^{m+\kappa ,p}/\mathcal{P}_{m+\kappa -1}.$ Since $\kappa <0$ (hence $%
\mathcal{P}_{\kappa -1}=\{0\}$), it remains to show that $A$ maps $%
D^{m+\kappa ,p}$ onto $D^{\kappa ,p}.$

Set $\kappa =-\ell $ with $\ell \in \Bbb{N},$ so that, by (\ref{18}), every $%
f\in D^{\kappa ,p}=D^{-\ell ,p}$ has the form $f=\sum_{|\alpha |_{1}=\ell
}\partial ^{\alpha }f_{\alpha }$ for some $f_{\alpha }\in L^{p}.$ By Lemma 
\ref{lm10}, there is $v_{\alpha }\in D^{m,p}$ such that $Av_{\alpha
}=f_{\alpha }.$ Thus, if $u:=\sum_{|\alpha |_{1}=\ell }\partial ^{\alpha
}v_{\alpha },$ then $u\in D^{m-\ell ,p}=D^{m+\kappa ,p}$ and $Au=f.$ This
completes the proof that $A$ is an isomorphism of $D^{m+\kappa ,p}/\mathcal{P%
}_{m+\kappa -1}$ onto $D^{\kappa ,p}/\mathcal{P}_{\kappa -1}$ for every $%
\kappa \in \Bbb{Z}.$

We now prove that $A$ maps $D^{m+\kappa ,p}$ onto $D^{\kappa ,p}.$ This was
just done above when $\kappa \leq 0.$ If $\kappa >0$ and $f\in D^{\kappa
,p}, $ the first part of the proof (or Lemma \ref{lm10}) ensures that there
are $\pi \in \mathcal{P}_{\kappa -1}$ and $u\in D^{m+\kappa -1,p}$ such that 
$Au=f+\pi .$ Thus, it suffices to show that there is $\varpi \in \mathcal{P}
_{m+\kappa -1}$ such that $A\varpi =\pi ,$ for then $u-\varpi \in
D^{m+\kappa ,p}$ and $A(u-\varpi )=f.$

The dimension of the space of homogeneous $A$-harmonic polynomials of degree 
$\ell ,$ as calculated by Horv\'{a}th \cite{Hor58}, is $\nu (\ell ,N)-\nu
(\ell -m,N)$ with $\nu $ from (\ref{17}), where $\nu (\ell -m,N):=0$ if $%
\ell <m.$ Thus, the subspace of $A$-harmonic polynomials in $\mathcal{P}
_{m+\kappa -1}$ has dimension $\sum_{\ell =0}^{m+\kappa -1}\nu (\ell
,N)-\sum_{\ell =0}^{\kappa -1}\nu (\ell ,N).$ Since $\nu (\ell ,N)$ is also
the dimension of the space of homogeneous polynomials of degree $\ell ,$
this is just $\dim \mathcal{P}_{m+\kappa -1}-\dim \mathcal{P}_{\kappa -1}.$
As a result, the rank of $A:\mathcal{P}_{m+\kappa -1}\rightarrow \mathcal{P}
_{\kappa -1}$ is $\dim \mathcal{P}_{\kappa -1}.$ Thus, $A$ maps $\mathcal{P}
_{m+\kappa -1}$ onto $\mathcal{P}_{\kappa -1}$ and the proof is complete.
\end{proof}

In Theorem \ref{th12}, the isomorphism property amounts to the generalized
Calderon-Zygmund inequality (the reverse inequality is trivial) 
\begin{equation*}
||\,[u]_{m+\kappa -1}||_{D^{m+\kappa ,p}/\mathcal{P}_{m+\kappa -1}}\leq
C||[Au]_{\kappa -1}\,||_{D^{\kappa ,p}/\mathcal{P}_{\kappa -1}},
\end{equation*}
for $u\in D^{m+\kappa ,p}.$

We can now prove a sharper and more general form of Theorem \ref{thA}.

\begin{theorem}
\label{th13}Let $A$ denote the homogeneous elliptic operator (\ref{1}). If $%
u\in \mathcal{D}^{\prime }$ and $Au\in D^{\kappa ,p}$ for some integer $%
\kappa \geq -m$ and $1<p<\infty ,$ the following properties are equivalent:%
\newline
(i) $u\in D^{m+\kappa ,p}.$\newline
(ii) $u\in M_{0}^{s,p}$ for every $s>m+\kappa -N/p$ if $p>N$ and every $%
s>m+\kappa -1$ if $p\leq N.$\newline
(iii) $u\in M_{0}^{m+\kappa ,1}.$
\end{theorem}

\begin{proof}
(i) $\Rightarrow $ (ii) by Theorem \ref{th8}.

(ii) $\Rightarrow $ (iii) by letting $s=m+\kappa $ in (ii) and by using $%
M_{0}^{m+\kappa ,p}\subset M_{0}^{m+\kappa ,1}.$

(iii) $\Rightarrow $ (i). Assume $u\in M_{0}^{m+\kappa ,1}$ and $Au\in
D^{\kappa ,p}.$ By Theorem \ref{th12}, there is $v\in D^{m+\kappa ,p}$ such
that $Av=Au$ and, by (i) $\Rightarrow $ (iii) already proved above, $v\in
M_{0}^{m+\kappa ,1}.$ Hence, $u-v\in M_{0}^{m+\kappa ,1}.$ Since $A(u-v)=0,$
Corollary \ref{cor5} yields $u-v\in \mathcal{P}_{m+\kappa -1},$ so that $%
u=v+(u-v)\in D^{m+\kappa ,p}.$
\end{proof}

A straightforward corollary of Theorem \ref{th13} addresses the same issue
when $D^{m+\kappa ,p}$ is replaced with $W^{m+\kappa ,p}.$

\begin{corollary}
\label{cor14}Let $A$ denote the homogeneous elliptic operator (\ref{1}) and
let $\kappa \geq -m$ be an integer. Then, $u\in W^{m+\kappa ,p}$ if and only
if \newline
(i) $Au\in D^{\kappa ,p}$ and $u\in L^{p}$\newline
or\newline
(ii) $Au\in D^{\kappa ,p}\cap D^{-m,p}$ and $u\in M_{0}^{0,1}.$\newline
In particular, if $m<N,Au\in D^{\kappa ,p}\cap L^{Np/(N+mp)}$ with $%
N/(N-m)<p<\infty $ and $u\in M_{0}^{0,1},$ then $u\in W^{m+\kappa ,p}.$
\end{corollary}

\begin{proof}
In both (i) and (ii), the necessity is trivial and the ``in particular''
part follows from (ii) and Lemma \ref{lm15} below (with an independent
proof). To prove the sufficiency of (i), just use $L^{p}=M^{-N/p,p}\subset
M_{0}^{m+\kappa ,1}$ since $m+\kappa \geq 0>-N/p$ to get $u\in D^{m+\kappa
,p}$ by (iii) $\Rightarrow $ (i) of Theorem \ref{th13}. Thus, $u\in
L^{p}\cap D^{m+\kappa ,p}=W^{m+\kappa ,p}$ (\cite[Theorem 4.10, p. 337]
{BeSh88}).

In (ii), use (iii) $\Rightarrow $ (i) of Theorem \ref{th13} with $\kappa =-m$
to get $u\in L^{p},$ so that the result follows from (i).
\end{proof}

Part (i) is trivial if $\kappa =-m,$ or if $-m<\kappa \leq 0$ and $A-z$ is
an isomorphism of $W^{m+\kappa ,p}$ to $W^{\kappa ,p}$ for some $z\in \Bbb{C}
$ (e.g., $A=\Delta $ or, more generally, $A$ strongly elliptic), but the
latter can only happen if $A(\xi )-z\neq 0$ for every $\xi \in \Bbb{R}^{N},$
which need not hold for any $z$ if $A$ is merely elliptic. The simplest
counter-examples are given by the powers $\overline{\partial }^{m}$ of the
Cauchy-Riemann operator $\overline{\partial }$ when $N=2.$ On the other
hand, if $A$ is strongly elliptic, part (i) remains true with $D^{\kappa ,p}$
replaced with $W^{\kappa ,p},$ which is more general when $\kappa <0.$

When $\kappa >0,$ part (i) does not follow right away from classical
elliptic theory even if $A=\Delta .$ Indeed, if $\Delta u=f\in D^{\kappa ,p}$
and $u\in L^{p},$ then $\Delta u-u=f-u$ but since (\textit{a priori}) $f-u$
need not be in $W^{\kappa ,p},$ the regularity properties of $\Delta -1$ do
not yield $u\in W^{m+\kappa ,p}.$ In fact, they do, but only with extra work
(differentiation and bootstrapping; details are left to the reader) and the
result cannot be called well-known.

Since part (ii) involves the space $M_{0}^{0,1},$ it is new irrespective of $%
A.$ Recall that $u\in M_{0}^{0,1}$ is much more general than the necessary $%
u\in L^{p};$ see the examples of Section \ref{spaces}, but more information
than just $Au\in D^{\kappa ,p}$ is needed. A nonstandard example of (i) and
(ii) with $N=2$ and $m=1$ is that $u\in W^{1,p}$ if either $u\in L^{p}$ and $%
\overline{\partial }u\in L^{p}$ with $1<p<\infty ,$ or $u\in M_{0}^{0,1}$
and $\overline{\partial }u\in L^{p}\cap L^{2p/(p+2)}$ with $2<p<\infty .$ Of
course, $u=0$ if $\overline{\partial }u=0$ in both cases, consistent with
Corollary \ref{cor5}. 

We complete this section with four very different applications of Theorem 
\ref{th13}. We begin with a ``consistency'' property. These properties are
very useful, but not trivial in scales of spaces which are not ordered by
inclusion.

\begin{example}
\label{ex5} Suppose that $u\in D^{m+\kappa _{1},p_{1}}$ for some $\kappa
_{1}\in \Bbb{Z}$ and some $1<p_{1}<\infty .$ If $Au\in D^{\kappa _{2},p_{2}}$
for some $\kappa _{2}\geq \kappa _{1}$ and some $1<p_{2}<\infty ,$ then $%
u\in D^{m+\kappa _{2},p_{2}}.$ If $\kappa _{1}\geq -m,$ this follows at once
from (i) $\Leftrightarrow $ (iii) in Theorem \ref{th13} and from $%
M_{0}^{m+\kappa _{1},1}\subset M_{0}^{m+\kappa _{2},1}.$ If $\kappa _{1}<-m,$
choose $\ell \in \Bbb{N}$ such that $\kappa _{1}\geq -\ell m.$ By Theorem 
\ref{th12} for $A^{\ell -1},$ there is $v\in D^{\ell m+\kappa _{1},p_{1}}$
such that $A^{\ell -1}v=u.$ Hence, $A^{\ell }v=Au\in D^{\kappa _{2},p_{2}}.$
Since $\kappa _{2}\geq \kappa _{1}\geq -\ell m,$ the first step yields $v\in
D^{\ell m+\kappa _{2},p_{2}},$ whence $u=A^{\ell -1}v\in D^{m+\kappa
_{2},p_{2}}.$ This argument shows that a general result may be useful even
if a single operator is of interest.
\end{example}

When $\kappa \geq 0,$ the use of Theorem \ref{th13} is simplified in
problems $Au=G(u):$

\begin{example}
\label{ex6}Let $G:\Bbb{C}\rightarrow \Bbb{C}$ be a continuous function such
that $\lim_{|z|\rightarrow \infty }|G(z)|=\infty .$ If $u\in L_{loc}^{1}$
and $Au=G(u)\in D^{\kappa ,p}$ for some $\kappa \in \Bbb{N}_{0}$ and $1<p<%
\infty ,$ then $u\in D^{m+\kappa ,p}.$ By Theorem \ref{th13}, it suffices to
prove that $u\in M_{0}^{m+\kappa ,1}.$ In fact, we claim that $u\in
M^{\kappa ,1},$ which is stronger since $m>0.$ To see this, let $g:[0,\infty %
)\rightarrow [0,\infty )$ be defined by $g(t):=\min_{\theta \in [0,2\pi %
]}|G(te^{i\theta })|.$ Then, $g\geq 0$ is continuous and $\lim_{t\rightarrow 
\infty }g(t)=\infty .$ Let $h\geq 0$ be a continuous function on $[0,\infty )
$ such that $h\leq g,h(0)=0$ and $\lim_{t\rightarrow \infty }h(t)=\infty .$
The existence of $h$ is not an issue. Then, the convex hull $\Phi $ of $h$
is a Young function.\newline
By Theorem \ref{th8}, $G(u)\in D^{\kappa ,p}\subset M_{0}^{\kappa ,p}\subset
M_{0}^{\kappa ,1}.$ Since $0\leq \Phi (|u|)\leq |G(u)|,$ we infer that $\Phi
(|u|)\in M_{0}^{\kappa ,1}.$ Let $t_{0}\geq 0$ and $\lambda >0$ be chosen as
in Example \ref{ex2} when $q=1,$ so that $t\leq \Phi (\lambda t)$ if $t\geq %
t_{0}.$ Equivalently, $t\leq \lambda \Phi (t)$ if $t\geq \lambda t_{0}.$
Thus, $|u|\leq \lambda t_{0}+\lambda \Phi (|u|)\in M^{\kappa ,1},$ whence $%
u\in M^{\kappa ,1}$ and even $u\in M_{0}^{\kappa ,1}$ if $\kappa >0$
(because $\lambda t_{0}\in M^{0,1}\subset M_{0}^{\kappa ,1}$) or if $1\notin
L_{\Phi }$ (because $\lambda t_{0}>0$ can be chosen arbitrarily small if $%
1\notin L_{\Phi };$ see part (iv) of Example \ref{ex2}).
\end{example}

It is easy to generalize Example \ref{ex6} when $G=G(x,u,\nabla u,...,\nabla
^{k}u)$ (possibly $k>m$) and there is a Young function $\Phi $ such that $%
|G(x,u,\nabla u,...,\nabla ^{k}u)|\geq \Phi (|\nabla ^{j}u|)$ for some $%
0\leq j\leq m-1.$ Indeed, by the same argument, $\Phi (|\nabla ^{j}u|)\in
M_{0}^{\kappa ,1}$ implies $|\nabla ^{j}u|\in M^{\kappa ,1}$ and then $u\in
M^{j+\kappa ,1}\subset M_{0}^{m+\kappa ,1}$ by Theorem \ref{th7}. This also
works if $j=m$ and either $\kappa >0$ or $1\notin L_{\Phi },$ for then $%
|\nabla ^{m}u|\in M_{0}^{\kappa ,1}$ and so $u\in M_{0}^{m+\kappa ,1}$ by
Theorem \ref{th7}. If $N>1,$ this is not applicable when $G$ is linear in $%
(u,\nabla u,...,\nabla ^{k}u),$ unless $k=0.$

The next example shows how the properties of the spaces $M^{s,q},$ notably
Theorems \ref{th2}, \ref{th7} and \ref{th8}, can be combined with Theorem 
\ref{th13} to convert growth assumptions on the coefficients into regularity
results for the solutions. Similar issues have been discussed by various
authors; see \cite{MePaRaSc10}, \cite{Ra05} and the references therein, but
it is safe to say that the results given in Example \ref{ex7} below cannot
be proved by previously known arguments.

The following equivalent dual form of Sobolev's inequality will be useful.

\begin{lemma}
\label{lm15}If $k<N$ is a positive integer and if $N/(N-k)<p<\infty ,$ then $%
Np/(N+kp)>1$ and $L^{Np/(N+kp)}\hookrightarrow D^{-k,p}$ (dense embedding).
\end{lemma}

\begin{proof}
Let $q>1$ be such that $kq<N.$ By Sobolev's inequality, there is a constant $%
C>0$ independent of $\varphi \in \mathcal{C}_{0}^{\infty }$ such that $%
||\varphi ||_{q^{*k}}\leq C||\,|\nabla ^{k}\varphi |\,||_{q},$ where $%
q^{*k}:=Nq/(N-kq).$ By the denseness of $\mathcal{C}_{0}^{\infty }$ in $\dot{
D}^{k,q}$ (Remark \ref{rm2}), this yields the embedding\footnote{%
Explicitly, this embedding is given by $[u]_{k-1}\mapsto u-\pi _{u},$ where $%
\pi _{u}\in \mathcal{P}_{k-1}$ is the only polynomial such that $u-\pi
_{u}\in L^{q^{*k}};$ clearly, $u-\pi _{u}$ is independent of
the representative $u.$} $\dot{D}^{k,q}\hookrightarrow L^{q^{*k}}.$ Since $%
\mathcal{C}_{0}^{\infty }\subset \dot{D}^{k,q}$ is dense in $L^{q^{*k}},$
this embedding is dense and so, by duality, $L^{(q^{*k})^{\prime
}}\hookrightarrow (\dot{D}^{k,q})^{*}=D^{-k,q^{\prime }}.$ The embedding is
dense since $\dot{D}^{k,q}$ is reflexive. Since $(q^{*k})^{\prime
}=Nq/((N+k)q-N),$ the result follows by letting $q=p^{\prime }.$
\end{proof}

\begin{example}
\label{ex7}All the functions are real-valued. Consider the problem $-\nabla 
\cdot (a\nabla u)+cu=f$ on $\Bbb{R}^{N},$ where $a,c>0$ are $\mathcal{C}^{%
\infty }$ (for simplicity) and satisfy the conditions (i) $a^{-1/2}\in D^{1,%
\infty }$ and (ii) $a^{-1}c^{1/2}\in L^{\infty }.$ \newline
Neither $a$ nor $c$ needs to be bounded or bounded below by a positive
constant but, since $a^{-1/2}\in M^{1,\infty }$ by (i) and Theorem \ref{th8}
(i), $a(x)$ cannot decay (pointwise) faster than $|x|^{-2}$ at infinity. The
function $c$ can decay arbitrarily fast but, by (ii), it cannot grow faster
than $a^{2}.$ \newline
It is readily checked that the space $V:=\{u\in \mathcal{D}^{\prime
}:a^{1/2}\nabla u\in (L^{2})^{N},c^{1/2}u\in L^{2}\}$ is a Hilbert space for
the inner product $\int_{\Bbb{R}^{N}}a\nabla u\cdot \nabla v+\int_{\Bbb{R}%
^{N}}cuv.$ Hence, if $f\in L^{2}(\Bbb{R}^{N};c^{-1/2}dx),$ which is
henceforth assumed, there is a unique $u\in V$ such that $-\nabla \cdot %
(a\nabla u)+cu=f.$ Equivalently, 
\begin{equation}
-\Delta u=(a^{-1}\nabla a)\cdot \nabla u-a^{-1}cu+a^{-1}f.  \label{19}
\end{equation}
The right-hand side is in $L^{2}.$ Indeed, $(a^{-1}\nabla a)\cdot \nabla
u=(a^{-3/2}\nabla a)\cdot (a^{1/2}\nabla u)\in L^{2}$ by (i) since $u\in V$
and $a^{-1}cu=a^{-1}c^{1/2}(c^{1/2}u)\in L^{2}$ and $%
a^{-1}f=a^{-1}c^{1/2}(c^{-1/2}f)\in L^{2}$ by (ii) since $u\in V$ and since $%
c^{-1/2}f\in L^{2}.$ We claim that $u\in D^{2,2}.$ By Theorem \ref{th13}, it
suffices to show that $u\in M_{0}^{2,1}.$ As noted above, $a^{-1/2}\in M^{1,%
\infty }.$ Since $a^{1/2}\nabla u\in (L^{2})^{N}$ and $L^{2}=M^{-N/2,2},$ it
follows from Theorem \ref{th2} that $\nabla u=(a^{-1/2})a^{1/2}\nabla u\in
(M^{1-N/2,2})^{N}\subset (M^{1/2,1})^{N}.$ By Theorem \ref{th7}, $u\in
M^{3/2,1}\subset M_{0}^{2,1}.$\newline
\qquad Assume now $N>2$ and replace (i) and (ii) with (i') $a^{-1/2}\in
D^{1,N}$ and (ii') $a^{-1}c^{1/2}\in L^{N}.$ By using $u\in V$ and (i'), $%
(a^{-1}\nabla a)\cdot \nabla u=(a^{-3/2}\nabla a)\cdot (a^{1/2}\nabla u)\in
L^{2N/(N+2)}.$ By Lemma \ref{lm15} (with $k=1,p=2$ and since $N/(N-1)<2$
when $N>2$), it follows that $(a^{-1}\nabla a)\cdot \nabla u\in D^{-1,2}.$
Likewise, by using $u\in V$ and (ii'), $a^{-1}cu\in D^{-1,2}$ and $%
a^{-1}f\in D^{-1,2}.$ Thus, the right-hand side of (\ref{19}) is in $%
D^{-1,2}.$ We claim that $u\in M_{0}^{1,1},$ so that $u\in D^{1,2}$ by
Theorem \ref{th13}. First, $a^{-1/2}\in M_{0}^{s,N}$ for every $s>0$ by
Theorem \ref{th8} (ii). In particular, $a^{-1/2}\in M_{0}^{N/2,N}.$ Next, $%
a^{1/2}\nabla u\in (L^{2})^{N}=(M^{-N/2,2})^{N}.$ Hence, by Theorem \ref{th2}%
, $\nabla u=a^{-1/2}(a^{1/2}\nabla u)\in (M_{0}^{0,2N/(N+2)})^{N}\subset
(M_{0}^{0,1})^{N}$ and so $u\in M_{0}^{1,1}$ by Theorem \ref{th7}.\newline
Let $2^{*}:=2N/(N-2)$ (recall $N>2$). By the Sobolev inequality, $\nabla
u-U_{\infty }\in (L^{2^{*}})^{N}$ with $U_{\infty }\in \Bbb{R}^{N}$ if (i)
and (ii) hold and $u-u_{\infty }\in L^{2^{*}}$ with $u_{\infty }\in \Bbb{R}$
if (i') and (ii') hold. In particular, $U_{\infty }=0$ if (i), (ii), (i')
and (ii') hold (because $u\in D^{1,2}$) and so $u\in D^{1,2^{*}}.$ With
this, it is not hard to find further conditions on $a,c$ and $f$ (compatible
with previous assumptions) ensuring that the right-hand side of (\ref{19})
is in $L^{2^{*}}.$ Then, $u\in D^{2,2^{*}}$ by Theorem \ref{th13} since $%
u\in M_{0}^{2,1}$ is already known. In turn, this implies $u-u_{\infty }\in
D^{2,2^{*}}\cap L^{2^{*}}=W^{2,2^{*}}.$
\end{example}

By a scaling argument, $\mathcal{C}_{0}^{\infty }\not\subset D^{-k,p}$ if $k%
\geq N$ and $1<p<\infty $ or $0<k<N$ and $1<p<N/(N-k).$ Example \ref{ex8}
below shows that the latter result, extended to the optimal\footnote{%
Since we did not define $D^{-k,p}$ when $p=1.$} range $1<p\leq N/(N-k),$ can
be derived from Theorem \ref{th13}.

\begin{example}
\label{ex8}Let $u$ be a smooth function equal to $|x|^{2-N}$ for $|x|$ large
enough if $N>2,$ or equal to $\log |x|$ for $|x|$ large enough if $N=2.$
Then, $u\in D^{1,p}$ with $p>N/(N-1)$ and so $u\in M_{0}^{1,p}\subset
M_{0}^{1,1}$ by Theorem \ref{th8}. Therefore, $\Delta u\notin D^{-1,p}$ if $%
1<p\leq N/(N-1),$ for otherwise $u\in D^{1,p}$ by Theorem \ref{th13}, which
is obviously false. Since $\Delta u\in \mathcal{C}_{0}^{\infty },$ this
shows that $\mathcal{C}_{0}^{\infty }\not\subset D^{-1,p}$ if $N>1$ and $1<p%
\leq N/(N-1).$ Likewise, $\mathcal{C}_{0}^{\infty }\not\subset D^{-2,p}$ if 
$N>2$ and $1<p\leq N/(N-2)$ because $u\in L^{p}$ if and only if $p>N/(N-2).$
More generally, $\mathcal{C}_{0}^{\infty }\not\subset D^{-k,p}$ when $N>k>0$
and $1<p\leq N/(N-k)$ can be seen by using the function $|x|^{2\ell -N}$ and
the operator $\Delta ^{\ell }$ with $\ell =k/2$ when $k$ is even or $\ell
=(k+1)/2$ when $k$ is odd. By Lemma \ref{lm15}, these non-embeddings are
sharp.
\end{example}

\section{Exterior domains\label{exterior}}

In this section, $\Omega \subset \Bbb{R}^{N}$ is an exterior domain (i.e., $%
\Bbb{R}^{N}\backslash \Omega $ is compact). To fix ideas, we shall also
assume that $0\notin \overline{\Omega }.$ In particular, $\Omega \neq \Bbb{R}%
^{N}$ but also $\Omega \neq \Bbb{R}^{N}\backslash \{0\}.$ We shall extend
Theorem \ref{th13} to this setting, but unexpected necessary restrictions on 
$N$ and $p$ arise when $\kappa <0,$ which are not needed when $\Omega =\Bbb{R%
}^{N}.$

If $k\in \Bbb{N}_{0}$ and $1<p<\infty ,$ the homogeneous Sobolev space $%
D^{k,p}(\Omega )$ is defined by (\ref{3}) after merely replacing $\Bbb{R}%
^{N} $ with $\Omega .$ If $\ell \in \Bbb{N},$ the space $D^{-\ell ,p}(\Omega
)$ is the dual of the completion $D_{0}^{\ell ,p^{\prime }}(\Omega )$ of $%
\mathcal{C}_{0}^{\infty }(\Omega )$ for the norm $||\,|\nabla ^{\ell
}\varphi |\,||_{p^{\prime },\Omega }.$ (If $\Omega =\Bbb{R}^{N},$ the
definition (\ref{16}) is recovered since $\mathcal{C}_{0}^{\infty }$ is
dense in $\dot{D}^{\ell ,p^{\prime }}.$) With this definition, $D^{-\ell
,p}(\Omega )$ is a space of distributions, $\partial ^{\alpha }$ maps $%
D^{k,p}(\Omega )$ into $D^{k-|\alpha |_{1},p}(\Omega )$ for every $k\in \Bbb{%
\ Z}$ and $\partial ^{\alpha }$ is continuous from $D^{k,p}(\Omega )/%
\mathcal{P}_{k-1}$ to $D^{k-|\alpha |_{1},p}(\Omega )/\mathcal{P}_{k-|\alpha
|_{1}-1}.$ For more details, see e.g. \cite{Ga11}.

The first task will be to adjust Theorem \ref{th13} to the new setting
(Theorem \ref{th18}). To begin with, spaces $M^{s,p}(\Omega )$ and $%
M_{0}^{s,p}(\Omega )$ can also be defined on $\Omega $ in the obvious way,
by merely replacing $B_{R}$ with 
\begin{equation}
\Omega _{R}:=B_{R}\cap \Omega ,  \label{20}
\end{equation}
in (\ref{6}) and (\ref{7}) and by choosing $R>0$ large enough that $|\Omega
_{R}|>0.$ However, to ensure the $L^{p}$-integrability on $\Omega _{R},$ the
definition of $M^{s,p}(\Omega )$ must incorporate $u\in L_{loc}^{p}(%
\overline{\Omega })$ rather than just $u\in L_{loc}^{p}(\Omega ).$ This is
of course immaterial when $\Omega =\Bbb{R}^{N}.$

\begin{remark}
\label{rm5} The extension by $0$ outside $\Omega $ maps $M^{s,p}(\Omega )$ ($%
M_{0}^{s,p}(\Omega )$) into $M^{s,p}$ ($M_{0}^{s,p}$). As a result, Theorems 
\ref{th1}, \ref{th2} and \ref{th3} have obvious generalizations that we
shall not spell out explicitly.
\end{remark}

The aforementioned possible restrictions about $N$ and $p$ originate in part
(i) of the following lemma, related to Lemma \ref{lm15}.

\begin{lemma}
\label{lm16}Let $\omega \subset \Bbb{R}^{N}$ be an open subset. \newline
(i) Let $\varphi \in \mathcal{C}^{\infty }$ be such that $\limfunc{Supp}%
\varphi \subset \omega $ and that $\limfunc{Supp}\nabla \varphi $ is
compact. Then $\varphi v\in D^{k,p}$ for every $v\in D^{k,p}(\omega )$ if
either $k\geq 0$ and $1\leq p\leq \infty $ or $-N<k<0$ and $N/(N+k)<p<\infty %
.$ \newline
(ii) If $\omega $ is bounded and $\partial \omega $ has the cone property,
then $D^{k,p}(\omega )=W^{k,p}(\omega )$ if either $k\geq 0$ and $1\leq p%
\leq \infty $ or $k<0$ and $1<p<\infty .$ \newline
\end{lemma}

\begin{proof}
(i) If $k\geq 0,$ this follows from Leibnitz' rule and from $D^{k,p}(\omega
)\subset W_{loc}^{k,p}(\omega )$ (use $\limfunc{Supp}\nabla \varphi $
compact; in particular, $\varphi $ is locally constant outside a ball and
therefore bounded). Below, we give a proof when $k=-1$ (hence $N>1$ and $%
N/(N-1)<p<\infty $). When $k<0$ is arbitrary, the modifications are
straightforward.

First, $p>N/(N-1)$ amounts to $p^{\prime }<N,$ so that $p^{\prime
*}:=Np^{\prime }/(N-p^{\prime })<\infty .$ Let $\psi \in \mathcal{C}%
_{0}^{\infty }$ be given. Since $v\in D^{-1,p}(\omega ),$ it follows that $%
|\langle \varphi v,\psi \rangle |\leq C||\,|\nabla (\varphi \psi
)|\,||_{p^{\prime }}$ with $C>0$ independent of $\psi .$ Now, use $%
||\,|\nabla (\varphi \psi )|\,||_{p^{\prime }}\leq C_{\varphi }(||\,|\nabla
\psi |\,||_{p^{\prime }}+||\,\psi \,||_{p^{\prime },S_{\varphi }}),$ where $%
S_{\varphi }:=\limfunc{Supp}\nabla \varphi $ and $C_{\varphi }>0$ is
independent of $\psi .$ Next, by H\"{o}lder's inequality, $||\,\psi
\,||_{p^{\prime },S_{\varphi }}\leq |S_{\varphi }|^{1/N}||\,\psi
\,||_{p^{\prime *},S_{\varphi }},$ whereas $||\psi ||_{p^{\prime
*},S_{\varphi }}\leq ||\psi ||_{p^{\prime *}}\leq C||\,|\nabla \psi
|\,||_{p^{\prime }}$ by Sobolev's inequality. Altogether, $|\langle \varphi
v,\psi \rangle |\leq C||\,|\nabla \psi |\,||_{p^{\prime }}$ for every $\psi
\in \mathcal{C}_{0}^{\infty },$ whence $\varphi v\in D^{-1,p}.$

(ii) is trivial if $k=0$ and proved in \cite[p. 21]{Ma11} if $k>0.$ If so
and if $1<p<\infty ,$ then $D_{0}^{k,p}(\omega )=W_{0}^{k,p}(\omega )$ with
equivalent norms (Poincar\'{e}'s inequality), so that $D^{-k,p^{\prime
}}(\omega )=W^{-k,p^{\prime }}(\omega )$ by duality. Exchange the roles of $%
p $ and $p^{\prime }$ and of $k$ and $-k$ to get $D^{k,p}(\omega
)=W^{k,p}(\omega )$ when $k<0.$
\end{proof}

In part (i), the restrictions on $N$ and $p$ when $k<0$ are needed even if $%
\omega $ is bounded. In particular, if $v\in D^{k,p}(\omega )$ has compact
support, the extension of $v$ by $0$ need \emph{not} be in $D^{k,p}$ without
these restrictions; see Example \ref{ex8} and preceding comments (indeed, $%
\mathcal{C}_{0}^{\infty }(\omega )\subset D^{k,p}(\omega )$ for every $k\in 
\Bbb{Z}$ and $1<p<\infty $ when $\omega $ is bounded). On the other hand, no
restriction on $N$ and $p$ is necessary if $\omega $ is bounded and $\Bbb{R}%
^{N}$ is replaced with a \emph{bounded} open subset $\widetilde{\omega }%
\supset $ $\omega ,$ because Poincar\'{e}'s inequality can be substituted
for Sobolev's inequality in the proof.

The following generalization of Theorem \ref{th8} is straightforward. 

\begin{lemma}
\label{lm17}If $\partial \Omega $ has the cone property, Theorem \ref{th8}
remains true upon replacing $D^{k,p}$ and $M^{s,p}$ with $D^{k,p}(\Omega )$
and $M^{s,p}(\Omega ),$ respectively.
\end{lemma}

\begin{proof}
Let $R_{0}>0$ be large enough that $\Bbb{R}^{N}\backslash \Omega \subset
B_{R_{0}}.$ If $u\in D^{k,p}(\Omega ),$ then $u\in D^{k,p}(\Omega
_{R_{0}})=W^{k,p}(\Omega _{R_{0}})$ (Lemma \ref{lm16} (ii)) and so $u\in
L_{loc}^{p}(\overline{\Omega }).$ Let $\varphi \in \mathcal{C}^{\infty
}(\Omega )$ be such that $\varphi =1$ outside $B_{R_{0}}$ and $\varphi =0$
on some neighborhood of $\Bbb{R}^{N}\backslash \Omega .$ By Lemma \ref{lm16}
(i) (with no restriction on $N$ or $p$ since $k\geq 0$), $\varphi u\in
D^{k,p}$ and so, by Theorem \ref{th8} for $\varphi u,$ it follows that $%
\varphi u\in M^{s,p}$ or $\varphi u\in M_{0}^{s,p}$for the specified values
of $s.$ This trivially implies $u\in M^{s,p}(\Omega )$ or $u\in
M_{0}^{s,p}(\Omega ),$ as the case may be.
\end{proof}

It is now easy to prove a variant of Theorem \ref{th13}.

\begin{theorem}
\label{th18}Suppose that $\partial \Omega \in \mathcal{C}^{0,1}$ and let $%
R_{0}>0$ be such that $\Bbb{R}^{N}\backslash \Omega \subset B_{R_{0}}.$ If $%
u\in \mathcal{D}^{\prime }(\Omega )$ and $Au\in D^{\kappa ,p}(\Omega )$ for
some integer $\kappa \geq \max \{-m,1-N\}$ and $\max \{1,N/(N+\kappa )\}<p<%
\infty ,$ the following properties are equivalent:\newline
(i) $u\in D^{m+\kappa ,p}(\Omega ).$ \newline
(ii) $u\in D^{m+\kappa ,p}(\Omega _{R_{0}})\cap $ $M_{0}^{s,p}(\Omega )$ for
every $s>m+\kappa -N/p$ if $p>N$ and every $s>m+\kappa -1$ if $p\leq -N.$%
\newline
(iii) $u\in D^{m+\kappa ,p}(\Omega _{R_{0}})\cap $ $M_{0}^{m+\kappa ,1}(%
\Omega ).$\newline
Furthermore, (i) $\Rightarrow $ (ii) if it is only assumed that $\partial 
\Omega $ has the cone property, $\kappa \geq -m$ and $1<p<\infty $ and (ii) $%
\Rightarrow $ (iii) is always true.
\end{theorem}

\begin{proof}
(i) $\Rightarrow $ (ii) $\Rightarrow $ (iii) If $u\in D^{m+\kappa ,p}(\Omega
),$ it is obvious that $u\in D^{m+\kappa ,p}(\Omega _{R_{0}}),$ whereas $%
u\in M_{0}^{s,1}(\Omega )$ for every $s>m+\kappa -N/p$ if $p>N$ and every $%
s>m+\kappa -1$ if $p\leq -N$ by Lemma \ref{lm17}. In particular, $u\in
M_{0}^{m+\kappa ,1}(\Omega )$ by Remark \ref{rm5} and Theorem \ref{th1}
(ii). This also proves the ``furthermore'' part.

(iii) $\Rightarrow $ (i) Suppose that $u\in D^{m+\kappa ,p}(\Omega
_{R_{0}}), $ that $Au\in D^{\kappa ,p}(\Omega )$ and that $u\in
M_{0}^{m+\kappa ,1}(\Omega ).$ By Lemma \ref{lm16} (ii), $D^{m+\kappa
,p}(\Omega _{R_{0}})=W^{m+\kappa ,p}(\Omega _{R_{0}})$ and so, by the Stein
extension theorem (this uses $\partial \Omega \in \mathcal{C}^{0,1};$ see 
\cite[p. 154]{AdFo03}, \cite[Chapter 6]{St70}), $u$ can be extended to all
of $\Bbb{R}^{N} $ as a function $\widetilde{u}\in W^{m+\kappa
,p}(B_{R_{0}}). $ In particular, $A\widetilde{u}\in D^{\kappa ,p}(B_{R_{0}})$
and $A\widetilde{u}\in D^{\kappa ,p}(\Omega )$ since $\widetilde{u}=u$ on $%
\Omega . $

Choose $\varphi \in \mathcal{C}_{0}^{\infty }(B_{R_{0}})$ with $\varphi =1$
on a neighborhood of $\Bbb{R}^{N}\backslash \Omega $ and write $A\widetilde{u%
}=\varphi A\widetilde{u}+(1-\varphi )A\widetilde{u}.$ By Lemma \ref{lm16}
(i) with $\omega =B_{R_{0}}$ and, next, $\omega =\Omega ,$ we get $\varphi A%
\widetilde{u}\in D^{\kappa ,p}$ and $(1-\varphi )A\widetilde{u}\in D^{\kappa
,p}.$ This shows that $A\widetilde{u}\in D^{\kappa ,p}.$ Since $u\in
M_{0}^{m+\kappa ,1}(\Omega )$ implies $\widetilde{u}\in M_{0}^{m+\kappa ,1},$
Theorem \ref{th13} yields $\widetilde{u}\in D^{m+\kappa ,p},$ whence $u\in
D^{m+\kappa ,p}(\Omega ).$
\end{proof}

If $N>2$ and $u(x):=|x|^{2-N},$ then $\Delta u=0$ in $\Omega ,u\in
L^{p}(\Omega _{R_{0}})$ for every $1<p<\infty $ and $u\in L^{p}(\Omega )$ if
and only if $p>N/(N-2).$ In particular, $u\in M_{0}^{0,1}(\Omega ).$ Thus,
the hypotheses of Theorem \ref{th18} are satisfied with $m=2,\kappa =-2$ and 
$N/(N-2)<p<\infty .$ Since $u\notin L^{p}(\Omega )$ when $p\leq N/(N-2),$
this shows that the condition $p>N/(N-2)$ cannot be dropped. Of course, the
similarities with Example \ref{ex8} are no coincidence and the functions and
operators of that example also show that, more generally, $p>N/(N+\kappa )$
cannot be dropped in Theorem \ref{th18}.

We shall not spell out the obvious analog of Corollary \ref{cor14} (just
note that $L^{p}(\Omega )\cap D^{m+\kappa ,p}(\Omega )=W^{m+\kappa
,p}(\Omega )$ when $\kappa \geq -m$ follows, by extension, from the same
property when $\Omega =\Bbb{R}^{N}$ and from Lemma \ref{lm16} (ii)). The
consistency question, similar to Example \ref{ex5} when $\Omega =\Bbb{R}%
^{N}, $ is settled in the following corollary, but in a (necessarily) more
limited setting. It has not been addressed in works discussing existence,
even when $A=\Delta $ (for instance, \cite{SiSo96}).

\begin{corollary}
\label{cor19}Suppose that $u\in D^{m+\kappa _{1},p_{1}}(\Omega )$ for some
integer $\kappa _{1}\geq \max \{-m,-N+1\}$ and some $\max \{1,N/(N+\kappa
_{1})\}<p_{1}<\infty $ and that $Au\in D^{\kappa _{2},p_{2}}(\Omega )$ for
some $\kappa _{2}\geq \kappa _{1}$ and some $\max \{1,N/(N+\kappa
_{2})\}<p_{2}<\infty .$\newline
(i) If $\partial \Omega \in \mathcal{C}^{0,1}$ and $u\in D^{m+\kappa
_{2},p_{2}}(\Omega _{R_{0}})$ with $R_{0}>0$ such that $\Bbb{R}%
^{N}\backslash \Omega \subset B_{R_{0}}$ (whence $u\in W^{m+\kappa
_{2},p_{2}}(\Omega _{R_{0}})$ by Lemma \ref{lm16} (ii)), then $u\in
D^{m+\kappa _{2},p_{2}}(\Omega ).$ \newline
(ii) If $\Omega ^{\prime }$ is an open subset such that $\overline{\Omega }%
^{\prime }\subset \Omega ,$ then $u\in D^{m+\kappa _{2},p_{2}}(\Omega %
^{\prime }).$
\end{corollary}

\begin{proof}
(i) By Theorem \ref{th18} with $\kappa =\kappa _{1}$ and $p=p_{1},$ $u\in
M_{0}^{m+\kappa _{1},1}(\Omega )\subset M_{0}^{m+\kappa _{2},1}(\Omega )$
and then $u\in D^{m+\kappa _{2},p_{2}}(\Omega )$ by Theorem \ref{th18} with $%
\kappa =\kappa _{2}$ and $p=p_{2}.$

(ii) After enlarging $\Omega ^{\prime },$ it is not restrictive to assume $%
\partial \Omega ^{\prime }\in \mathcal{C}^{0,1}.$ Let $\varphi \in \mathcal{C%
}_{0}^{\infty }(\Omega )$ be such that $\varphi =1$ on some open ball $%
B\subset \Omega .$ By Lemma \ref{lm16} (i), $\varphi Au\in D^{\kappa
_{2},p_{2}},$ so that, by Theorem \ref{th12}, there is $v\in D^{m+\kappa
_{2},p_{2}}$ such that $Av=\varphi Au.$ In particular, $A(v-u)=0$ on $B.$ By
hypoellipticity, $v-u\in \mathcal{C}^{\infty }(B),$ whence $u\in D^{m+\kappa
_{2},p_{2}}(B^{\prime })=W^{m+\kappa _{2},p_{2}}(B^{\prime })$ for every
ball $B^{\prime }\Subset B.$ This shows that $u\in W_{loc}^{m+\kappa
_{2},p_{2}}(\Omega )$ and, hence, that $u\in W^{m+\kappa _{2},p_{2}}(\Omega
_{R_{0}}^{\prime })\subset D^{m+\kappa _{2},p_{2}}(\Omega _{R_{0}}^{\prime
}) $ for every $R_{0}>0$ such that $\Bbb{R}^{N}\backslash \Omega ^{\prime
}\subset B_{R_{0}}.$ Thus, $u\in D^{m+\kappa _{2},p_{2}}(\Omega ^{\prime })$
by (i) with $\Omega $ replaced with $\Omega ^{\prime }$.
\end{proof}

In part (i) of Corollary \ref{cor19}, the condition $u\in D^{m+\kappa
_{2},p_{2}}(\Omega _{R_{0}})$ depends only upon the behavior of $u$ near $%
\partial \Omega .$ This may be provable by elliptic regularity arguments.
For instance, if $A$ is properly elliptic (hence $m=2\ell $ is even), $%
\kappa _{1}\geq 0$ (hence $\kappa _{2}\geq 0$), $\partial \Omega \in 
\mathcal{C}^{m+\kappa _{2}}$ and $\partial ^{j}u/\partial \nu ^{j}\in
W^{m+\kappa _{2}-j-1/p_{2},p_{2}}(\partial \Omega )$ for $0\leq j\leq \ell
-1,$ classical elliptic regularity yields $u\in W^{m+\kappa
_{2},p_{2}}(\Omega _{R_{0}})$ (note that $u\in W_{loc}^{m+\kappa
_{2},p_{2}}(\Omega )$ since $Au\in D^{\kappa _{2},p_{2}}(\Omega )\subset
W_{loc}^{\kappa _{2},p_{2}}(\Omega ),$ whence $\partial ^{j}u/\partial \nu
^{j}\in W^{m+\kappa _{2}-j-1/p_{2},p_{2}}(\partial \Omega _{R_{0}})$ for $%
0\leq j\leq \ell -1$). This remains true under much more general boundary
conditions on $\partial \Omega $ under suitable smoothness requirements; see 
\cite[Corollary 2.1]{Ra09} and ``obvious'' generalizations when $\kappa >0$
in that paper. If $\kappa _{1}<0,$ there may or may not be an elliptic
regularity result to answer the question.

We complete this section with an example showing how solutions of boundary
value problems on $\Omega $ can be found in the space $M_{0}^{0,1}(\Omega ).$
As pointed out in the Introduction, the functions of $M_{0}^{0,1}$ vanish at
infinity in a generalized (averaged) sense. On the exterior domain $\Omega ,$
this property is obviously shared by the functions of $M_{0}^{0,1}(\Omega ).$
In spite of having no direct connection with regularity, this short
digression is included since it involves the $M^{s,q}$ scale introduced
earlier, which has not been used elsewhere to discuss the asymptotic
behavior of solutions of PDEs. In the next theorem, $W_{loc}^{1,q}(\overline{%
\Omega })$ refers to the space of distributions $u\in \mathcal{D}^{\prime
}(\Omega )$ such that $\varphi u\in W^{1,q}(\Omega )$ for every $\varphi \in 
\mathcal{C}_{0}^{\infty }(\overline{\Omega }).$

\begin{theorem}
\label{th20}Suppose that $N>2$ and that $\partial \Omega \in \mathcal{C}^{1}.
$ If $|x|^{(N+2)-2N/q}f\in L^{q}(\Omega )$ and $g\in W^{1-1/q,q}(\partial 
\Omega )$ for some $1<q<\infty ,$ the Dirichlet boundary value problem 
\begin{equation}
\left\{ 
\begin{array}{c}
\Delta u=f\text{ in }\Omega , \\ 
u=g\text{ on }\partial \Omega ,
\end{array}
\right.   \label{21}
\end{equation}
has a solution $u\in W_{loc}^{1,q}(\overline{\Omega })\cap M_{0}^{0,1}(%
\Omega ).$ If $\partial \Omega \in \mathcal{C}^{0,1},$ this remains true
when $3/2\leq q\leq 3$ (and, more generally, when $3(\varepsilon
+2)^{-1}<q<3(1-\varepsilon )^{-1}$ for some $0<\varepsilon \leq 1$ depending
only upon $\Omega $).
\end{theorem}

\begin{proof}
We reformulate the problem (\ref{21}) with the help of the Kelvin transform
method (\cite{AxBoRa01}, \cite[Vol. 1]{DaLi90}). Denote by $\Omega ^{K}$ the
bounded open subset of $\Bbb{R}^{N}$ obtained by the inversion $x\mapsto
y:=|x|^{-2}x$ of $\Omega \cup \{\infty \}.$ The boundary $\partial \Omega
^{K}$ is the inverse of $\partial \Omega ,$ so that $\partial \Omega ^{K}\in 
\mathcal{C}^{1}$. If $h$ is a function defined on a subset of $\Bbb{R}
^{N}\backslash \{0\},$ set 
\begin{equation*}
h^{K}(y):=|y|^{2-N}h(|y|^{-2}y),
\end{equation*}
(Kelvin transform of $f$). Note that $(h^{K})^{K}=h.$

Therefore, $u$ is a solution of (\ref{21}) if and only if $u^{K}$ solves $%
\Delta u^{K}=|y|^{-4}f^{K}$ on $\Omega ^{K}\backslash \{0\}$ and $%
u^{K}=g^{K} $ on $\partial \Omega ^{K}.$ In particular, if $|y|^{-4}f^{K}$
can be extended as a distribution on $\Omega ^{K},$ solutions $%
u:=(u^{K})^{K} $ of (\ref{21}) can be found by solving the Dirichlet problem 
\begin{equation}
\left\{ 
\begin{array}{l}
\Delta u^{K}=|y|^{-4}f^{K}\text{ in }\Omega ^{K}, \\ 
u^{K}=g^{K}\text{ on }\partial \Omega ^{K}.
\end{array}
\right.  \label{22}
\end{equation}
The standing assumption $|x|^{(N+2)-2N/q}f\in L^{q}(\Omega )$ is equivalent
to $|y|^{-4}f^{K}\in L^{q}(\Omega ^{K}).$ Since $L^{q}(\Omega ^{K})\subset
W^{-1,q}(\Omega ^{K})$ and since $\partial \Omega ^{K}\in \mathcal{C}^{1}$
and $g^{K}\in W^{1-1/q,q}(\partial \Omega ^{K}),$ there is a unique solution 
$u^{K}\in W^{1,q}(\Omega ^{K})$ of (\ref{22}). By standard trace theorems,
this follows by reduction to the case $g^{K}=0.$ When $g^{K}=0,$ see Simader
and Sohr \cite[Theorem 1.2, p. 45]{SiSo96}, Morrey \cite[Remarks, p. 157]
{Mo66} (when $q\geq 2$) or Dautray and Lions \cite[p. 409]{DaLi90} (with
proof in \cite{LiMa61}) when $\partial \Omega ^{K}\in \mathcal{C}^{\infty }.$

By Jerison and Kenig \cite[Theorem 1.1]{JeKe95}, the case when $\partial
\Omega ^{K}\in \mathcal{C}^{0,1}$ (i.e., $\partial \Omega \in \mathcal{C}%
^{0,1}$) introduces the necessary restrictions $3(\varepsilon
+2)^{-1}<q<3(1-\varepsilon )^{-1}$ for some $0<\varepsilon \leq 1$ depending
only upon $\Omega ^{K}$ (i.e., upon $\Omega $). This range includes $3/2\leq
q\leq 3.$

Clearly, $u^{K}\in W^{1,q}(\Omega ^{K})$ implies $u\in W_{loc}^{1,q}(%
\overline{\Omega }),$ but this alone does not give any information about the
behavior of $u$ at infinity. Choose $R_{0}>0$ such that $\Bbb{R}%
^{N}\backslash \Omega \subset B_{R_{0}}.$ Then, $I_{0}:=\int_{\Omega
_{R_{0}}}|u|<\infty $ since $u\in W^{1,q}(\Omega _{R_{0}})$ and, if $R>R_{0}$
(observe that $-|y|^{-2N}$ is the Jacobian of $|y|^{-2}y$) 
\begin{equation*}
\left. 
\begin{array}{c}
\int_{\Omega _{R}}|u|=I_{0}+\int_{B_{R}\backslash
B_{R_{0}}}|u|=I_{0}+\int_{B_{R_{0}^{-1}}\backslash
B_{R^{-1}}}|y|^{-(N+2)}|u^{K}(y)|dy.
\end{array}
\right.
\end{equation*}
Now, by interior elliptic regularity, $u^{K}\in W_{loc}^{2,q}(\Omega
^{K})\subset W^{2,q}(B_{R_{0}^{-1}})$ and so, by H\"{o}lder's inequality, $%
\int_{B_{R_{0}^{-1}}\backslash B_{R^{-1}}}|y|^{-(N+2)}|u^{K}(y)|dy\leq
C_{r}R^{N+2-N/r^{\prime }}||u^{K}||_{r,B_{R_{0}^{-1}}}$ if $%
W^{2,q}(B_{R_{0}^{-1}})\hookrightarrow L^{r}(B_{R_{0}^{-1}}),$ where $%
C_{r}>0 $ is independent of $R.$ Thus, $u\in M^{2-N/r^{\prime },1}(\Omega )$
for every such $r.$ By the Sobolev embedding theorem, we can choose $%
r=Nq/(N-2q)$ if $q<N/2$ or $r$ arbitrarily large if $q\geq N/2.$ In summary, 
$u\in M^{s,1}\Omega )$ with $s=-N/q^{\prime }$ if $q<N/2$ or with $s>2-N$ if 
$q\geq N/2.$ Since $N>2$ and $q>1,$ it follows that $u\in M^{s,1}(\Omega )$
for some $s<0, $ whence $u\in M_{0}^{0,1}(\Omega ).$
\end{proof}

If $q>2N/(N+2),$ it is a bit tedious but not difficult to check that the
solution $u$ of Theorem \ref{th20} is even in $L^{p}$ for some $1<p<\infty $
(hence in $M_{0}^{0,1}$), but this is not the case if $1<q\leq 2N/(N+2).$
Also, $|u(x)|=O(|x|^{2-N})=o(1)$ for large $|x|$ if $u^{K}$ is continuous at
the origin. This requires $|x|^{(N+2)-2N/q}f\in L^{q}(\Omega )$ with $q>N/2$
(so that $u^{K}\in W_{loc}^{2,q}(\Omega ^{K})$), which may not be compatible
with $q<3(1-\varepsilon )^{-1}$ when $N\geq 7$ and $\partial \Omega \in 
\mathcal{C}^{0,1}.$ At any rate, this is stronger\footnote{%
Since the condition $|x|^{(N+2)-2N/q}f\in L^{q}(\Omega )$ is equivalent to $%
|y|^{-4}f^{K}\in L^{q}(\Omega ^{K})$ and $\Omega ^{K}$ is bounded, it
becomes more restrictive as $q$ is increased.} than $|x|^{N-2}f\in L^{N/2}(%
\Omega ).$ By comparison, Theorem \ref{th20} shows that when $\partial 
\Omega \in \mathcal{C}^{1},$ solutions vanishing at infinity still exist
under the (much) more general condition $|x|^{(N+2)-2N/q}f\in L^{q}(\Omega )$
for some $q>1,$ only slightly stronger than $|x|^{2-N}f\in L^{1}(\Omega ).$
For example, this amounts to $\alpha <-2$ versus $\alpha <-N$ if $%
f(x)=|x|^{\alpha }$ for large $|x|.$

The method of Theorem \ref{th20} can readily be used with other boundary
conditions and other operators. Neither the homogeneity nor the constancy of
the coefficients is important, as long as the problem on $\Omega ^{K}$ fits
within the elliptic theory.

\section{Systems\label{systems}}

In what follows, $n\in \Bbb{N},\mathbf{m}:=(m_{1},...,m_{n})\in (\Bbb{N}%
_{0})^{n}$ and $\mathbf{\kappa }:=(\kappa _{1},...,\kappa _{n})\in \Bbb{Z}%
^{n}$ are given and $\mathbf{A}:=(A_{jk})_{1\leq j,k\leq n}$ is a matrix
differential operator where 
\begin{equation*}
A_{jk}:=i^{m_{jk}}\sum_{|\alpha |_{1}=m_{jk}}a_{jk\alpha }\partial ^{\alpha
},
\end{equation*}
is homogeneous of order $m_{jk}:=m_{k}+\kappa _{k}-\kappa _{j},$ with the
understanding that $A_{jk}=0$ if $m_{jk}<0.$ With these assumptions, the $n$
-tuples $\mathbf{m}+\mathbf{\kappa }$ and $-\mathbf{\kappa }$ are a system
of DN numbers for the operator $\mathbf{A}$ (Douglis and Nirenberg \cite
{DoNi55}, Wloka \textit{et al.} \cite{WlRoLa95}).

Let $\mathbf{A}(\xi )$ denote the $n\times n$ matrix with entries 
\begin{equation*}
A_{jk}(\xi ):=\sum_{|\alpha |_{1}=m_{jk}}a_{jk\alpha }\xi ^{\alpha }.
\end{equation*}
Since $A_{jk}(\xi )$ is homogeneous of degree $m_{jk},$ it follows that $%
\det (\mathbf{A}(\xi ))$ is homogeneous of degree $M:=\sum_{j=1}^{n}m_{j}.$

We shall assume that $\mathbf{A}$ is DN elliptic. This means that 
\begin{equation*}
\det (\mathbf{A}(\xi ))\neq 0\text{ for every }\xi \in \Bbb{R}^{N}\backslash
\{0\}.
\end{equation*}
The above assumptions are unaffected by changing $\mathbf{\kappa }$ into $%
\mathbf{\kappa }+\iota \mathbf{1}$ where $\iota \in \Bbb{Z}$ and $\mathbf{1}%
:=(1,...,1)\in \Bbb{N}^{n}.$

Homogeneous Petrovsky-elliptic systems ($m_{1}=\cdots =m_{n}$ and $\kappa
_{1}=\cdots =\kappa _{n}$), such as the linear elasticity system and
diagonal systems of homogeneous elliptic operators (arbitrary $m_{j}$ and $%
\kappa _{j}$) are the simplest examples satisfying the above conditions. The
Stokes system, with $n=N+1$ and $m_{1}=\cdots =m_{N}=2,\kappa _{1}=\cdots
\kappa _{N}=\kappa \in \Bbb{Z}$ and $m_{N+1}=0,$ $\kappa _{N+1}=\kappa +1$
is a less obvious example.

Since the space of formal scalar differential operators with constant
coefficients is a commutative ring, $\det \mathbf{A}$ is defined as a scalar
differential operator with constant coefficients. This remark was first used
long ago by Malgrange \cite{Ma56} to prove the existence of a fundamental
solution for systems with constant coefficients. We shall use it in a
technically different way, but in a similar spirit, to generalize Theorem 
\ref{th12}. In practice, $\det \mathbf{A}$ is obtained by replacing $\xi
^{\alpha }$ with $i^{|\alpha |_{1}}\partial ^{\alpha }$ in $\det (\mathbf{A}
(\xi )),$ so that it is homogeneous of order $M$ and elliptic.

For simplicity of notation, we set 
\begin{equation*}
\left. 
\begin{array}{c}
D^{\mathbf{m+\kappa },p}:=\prod_{j=1}^{n}D^{m_{j}+\kappa _{j},p},\qquad 
\mathcal{P}_{\mathbf{m+\kappa -1}}:=\prod_{j=1}^{n}\mathcal{P}_{m_{j}+\kappa
_{j}-1} \\ 
D^{\mathbf{\kappa },p}:=\prod_{j=1}^{n}D^{\kappa _{j},p},\qquad \mathcal{P}_{%
\mathbf{\kappa -1}}:=\prod_{j=1}^{n}\mathcal{P}_{\kappa _{j}-1}.
\end{array}
\right.
\end{equation*}
Thus, if $\mathbf{u}=(u_{j})_{1\leq j\leq n}\in D^{\mathbf{\ m+\kappa },p}$ (%
$\mathbf{f}=(f_{j})_{1\leq j\leq n}\in D^{\mathbf{\ \kappa },p}$), the
equivalence class of $\mathbf{u}$ in $D^{\mathbf{m+\kappa },p}/\mathcal{P}_{%
\mathbf{m+k-1}}$ (of $\mathbf{f}$ in $D^{\mathbf{\kappa },p}/\mathcal{P}_{%
\mathbf{k-1}}$) is $[\mathbf{u}]_{\mathbf{m+\kappa -1}}=([u]_{m_{j}+\kappa
_{j}-1})_{1\leq j\leq n}$ ($[\mathbf{f}]_{\mathbf{\kappa -1}
}=([f_{j}]_{\kappa _{j}-1})_{1\leq j\leq n}$).

\begin{theorem}
\label{th21}If $1<p<\infty ,$ the operator $\mathbf{A}$ is a linear
isomorphism from $D^{\mathbf{m+\kappa },p}/\mathcal{P}_{\mathbf{m+k-1}}$
onto $D^{\mathbf{\kappa },p}/\mathcal{P}_{\mathbf{k-1}}$ and a homomorphism
of $D^{\mathbf{m+\kappa },p}$ onto $D^{\mathbf{\kappa },p}.$ \newline
\end{theorem}

\begin{proof}
A routine verification shows that $\mathbf{A}$ maps $D^{\mathbf{m+\kappa }
,p} $ into $D^{\mathbf{\kappa },p}$ and $\mathcal{P}_{\mathbf{m+\kappa -1}}$
into $\mathcal{P}_{\mathbf{\kappa -1}},$ so that $\mathbf{A}$ is well
defined from $D^{\mathbf{m+\kappa },p}/\mathcal{P}_{\mathbf{m+k-1}}$ to $D^{%
\mathbf{\kappa },p}/\mathcal{P}_{\mathbf{k-1}}.$ Furthermore, in that
setting, $\mathbf{A}$ is one-to-one, for if $\mathbf{u}\in D^{\mathbf{%
m+\kappa },p}$ and $\mathbf{Au}\in \mathcal{P}_{\mathbf{k-1}},$ the usual
Fourier transform argument shows that $\limfunc{Supp}\widehat{\mathbf{u}}
=\{0\}.$ Hence, the components $u_{j}\in D^{m_{j}+\kappa _{j},p}$ of $%
\mathbf{u}$ are polynomials and so $u_{j}\in \mathcal{P}_{m_{j}+\kappa
_{j}-1}$ by Lemma \ref{lm11}, which in turn yields $[\mathbf{u}]_{\mathbf{\
m+\kappa -1}}=[\mathbf{0}]_{\mathbf{\ m+\kappa -1}}.$

We now prove that $\mathbf{A}$ is onto $D^{\mathbf{\kappa },p}/\mathcal{P}_{%
\mathbf{k-1}}$ by exhibiting a right inverse. For every $1\leq j,k\leq n,$
denote by $C_{jk}$ the $(j,k)$ cofactor of $\mathbf{A}.$ This is the scalar
differential operator obtained by replacing $\xi ^{\alpha }$ with $%
i^{|\alpha |_{1}}\partial ^{\alpha }$ in the corresponding cofactor $%
C_{jk}(\xi )$ of $\mathbf{A}(\xi ).$ As a result, $C_{jk}$ is homogeneous of
order $M-m_{k}-\kappa _{k}+\kappa _{j}.$ In particular, $C_{k\ell }$
(homogeneous of order $M-m_{\ell }-\kappa _{\ell }+\kappa _{k}$) maps $%
D^{M+\kappa _{k},p}$ into $D^{m_{\ell }+\kappa _{\ell },p}$ and $\mathcal{P}
_{M+\kappa _{k}-1}$ into $\mathcal{P}_{m_{\ell }+\kappa _{\ell }-1}$ and so
it is a well defined operator from $D^{M+\kappa _{k},p}/\mathcal{P}
_{M+\kappa _{k}-1}$ to $D^{m_{\ell }+\kappa _{\ell },p}/\mathcal{P}_{m_{\ell
}+\kappa _{\ell }-1}.$ On the other hand, by the very definition of $C_{jk},$
\begin{equation}
\sum_{\ell =1}^{n}A_{j\ell }C_{k\ell }=\delta _{jk}\det \mathbf{A}\text{
(Kronecker delta).}  \label{23}
\end{equation}

It follows from Theorem \ref{th12} that $\det \mathbf{A}$ is an isomorphism
of $D^{M+\kappa _{k},p}/\mathcal{P}_{M+\kappa _{k}-1}$ onto $D^{\kappa
_{k},p}/\mathcal{P}_{\kappa _{k}-1}$ for $1\leq k\leq n.$ Denote by $B_{k}$
the inverse isomorphism and let $[f]_{\kappa _{k}-1}\in D^{\kappa _{k},p}/%
\mathcal{P}_{\kappa _{k}-1},$ so that $B_{k}[f]_{\kappa _{k}-1}\in
D^{M+\kappa _{k},p}/\mathcal{P}_{M+\kappa _{k}-1}.$ From the above, $%
C_{k\ell }B_{k}[f]_{\kappa _{k}-1}\in D^{m_{\ell }+\kappa _{\ell },p}/%
\mathcal{P}_{m_{\ell }+\kappa _{\ell }-1}$ and so $A_{j\ell }C_{k\ell
}B_{k}[f]_{\kappa _{k}-1}\in D^{\kappa _{j},p}/\mathcal{P}_{\kappa _{j}-1}$
since $A_{j\ell }$ (homogeneous of order $m_{\ell }+\kappa _{\ell }-\kappa
_{j}$) maps $D^{m_{\ell }+\kappa _{\ell },p}$ into $D^{\kappa _{j},p}$ and $%
\mathcal{P}_{m_{\ell }+\kappa _{\ell }-1}$ into $\mathcal{P}_{\kappa
_{j}-1}. $ Consequently, by (\ref{23}), $\sum_{\ell =1}^{n}A_{j\ell
}C_{k\ell }B_{k}[f]_{\kappa _{k}-1}=[0]_{\kappa _{j}-1}$ if $j\neq k$ and $%
\sum_{\ell =1}^{n}A_{k\ell }C_{k\ell }B_{k}[f]_{\kappa _{k}-1}=(\det \mathbf{%
A})B_{k}[f]_{\kappa _{k}-1}=[f]_{\kappa _{k}-1}.$ Therefore, the operator $%
\mathbf{B}:=(B_{jk})_{1\leq j,k\leq n}$ with $B_{jk}:=C_{kj}B_{k}$ acting
from $D^{\mathbf{\kappa },p}/\mathcal{P}_{\mathbf{k-1}}$ to $D^{\mathbf{\
m+\kappa },p}/\mathcal{P}_{\mathbf{m+k-1}}$ is a right inverse of $\mathbf{A}%
.$

To show that $\mathbf{A}$ maps $D^{\mathbf{m+\kappa },p}$ onto $D^{\mathbf{\
\kappa },p},$ recall that, by Theorem \ref{th12}, $\det \mathbf{A}$ maps $%
D^{M+\kappa _{k},p}$ onto $D^{\kappa _{k},p}$ for $1\leq k\leq n.$ Given $%
\mathbf{f}=(f_{k})_{1\leq k\leq n}\in D^{\mathbf{\kappa },p},$ choose $%
v_{k}\in D^{M+\kappa _{k},p}$ such that $(\det \mathbf{A})v_{k}=f_{k}$ and,
for $1\leq \ell \leq n,$ set $u_{\ell }:=\sum_{k=1}^{n}C_{k\ell }v_{k}.$
Then, $u_{\ell }\in D^{m_{\ell }+\kappa _{\ell },p}$ and, with $\mathbf{u}
:=(u_{\ell })_{1\leq \ell \leq n}\in D^{\mathbf{m+\kappa },p},$ we have $(%
\mathbf{Au})_{j}=\sum_{\ell =1}^{n}A_{j\ell }u_{\ell
}=\sum_{k=1}^{n}\sum_{\ell =1}^{n}A_{j\ell }C_{k\ell
}v_{k}=\sum_{k=1}^{n}\delta _{jk}(\det \mathbf{A})v_{k}=f_{j}.$ Thus, $%
\mathbf{Au}=\mathbf{f}$ and the proof is complete.
\end{proof}

When $\mathbf{A}$ is the Stokes system, partial results related to Theorem 
\ref{th21} have been proved, with the help of fundamental solutions, under
more restrictive assumptions about $\mathbf{f}$ (\cite{FaSo94}, \cite{Ga11}%
). In that regard, we point out that there are technical difficulties in
proving Theorem \ref{th21} in full generality based on the construction of a
suitable fundamental solution, as was done in Lemma \ref{lm10} in the scalar
case. Note that if $\mathbf{A}$ is the Stokes system, then $\det \mathbf{A}%
=(-1)^{N}\Delta ^{N}$ has order $2N.$

Upon using Theorem \ref{th21} instead of Theorem \ref{th12} in the proof, it
is now obvious how Theorem \ref{th13} can be generalized. It suffices to
introduce a convenient notation. If $\mathbf{a}=(a_{j})_{1\leq j\leq n}$ and 
$\mathbf{b}=(b_{j})_{1\leq j\leq n},$ the inequality $\mathbf{a}\geq \mathbf{%
\ b}$ ($\mathbf{a}>\mathbf{b}$) means $a_{j}\geq b_{j}$ ($a_{j}>b_{j}$) for $%
1\leq j\leq n.$ Also, if $\mathbf{s}=(s_{j})_{1\leq j\leq n},$ we set $%
M_{0}^{\mathbf{s},p}:=\prod_{j=1}^{n}M_{0}^{s_{j},p}.$ With this, we can
state:

\begin{theorem}
\label{th22}If $\mathbf{u}\in (\mathcal{D}^{\prime })^{n}$ and $\mathbf{Au}%
\in D^{\mathbf{\kappa },p}$ with $\mathbf{\kappa }\geq -\mathbf{m}$ and $1<p<%
\infty ,$ the following properties are equivalent:\newline
(i) $\mathbf{u}\in D^{\mathbf{m+k},p}.$\newline
(ii) $\mathbf{u}\in M_{0}^{\mathbf{s},p}$ for every $\mathbf{s}>\mathbf{m}+%
\mathbf{\kappa }-(N/p)\mathbf{1}$ if $p>N$ and every $\mathbf{s}>\mathbf{m}+%
\mathbf{k}-\mathbf{1}$ if $p\leq N.$\newline
(iii) $\mathbf{u}\in M_{0}^{\mathbf{m}+\mathbf{\kappa },1}.$\newline
\end{theorem}

\end{document}